\documentclass[sn-mathphys]{sn-jnl}
\usepackage{amsmath,amssymb,amsthm}
  \usepackage{graphicx,mathtools}
  \usepackage{xcolor,soul}
\usepackage{comment}
\usepackage{marginnote}
\usepackage{array}
\usepackage{booktabs}
\usepackage{mathtools}

\numberwithin{equation}{section}

\newtheorem{theorem}{Theorem}[section]
\newtheorem{cor}[theorem]{Corollary}
\newtheorem{lemma}[theorem]{Lemma}

\newtheorem{prop}[theorem]{Proposition} \theoremstyle{definition}
\newtheorem{definition}[theorem]{Definition}
\newtheorem{rem}[theorem]{Remark}

\theoremstyle{plain}

\newcommand{\bmat}{\left[\begin{matrix}}
\newcommand{\emat}{\end{matrix}\right]}
\newcommand{\col}[1]{\begin{bmatrix}#1\end{bmatrix}}

\def\beq{\begin{equation} } \def\eeq{\end{equation}}

\DeclarePairedDelimiter\abs\lvert\rvert

\def\RR{\mathbb R}

\def\ben{\begin{enumerate} }

\def\een{\end{enumerate} }

\def \R{ {\mathbb R}}
  
\def \p { \partial}
\def \l{ \langle }

\def \Id{ \text{Id} } 
\def \div{ \text{div}}

\def \Op{ Q} 

\renewcommand{\restriction}{\mathord{\upharpoonright}}

\def\trac{\mathcal{T}}

\newcommand\PS{{P\!/\!S}}

\jyear{2021}

\begin{document}
\title[Interface recovery from reflections]{Recovery of wave speeds and density of mass across a heterogeneous smooth interface from acoustic and elastic wave reflection operators}

\author[1]{\fnm{Sombuddha} \sur{Bhattacharyya}}\email{sombuddha@iiserb.ac.in}

\author[3]{\fnm{Maarten V.} \sur{de Hoop}}\email{mdehoop@rice.edu}

\author*[2]{\fnm{Vitaly} \sur{Katsnelson}}\email{vkatsnel@nyit.edu}

\author[1]{\fnm{Gunther} \sur{Uhlmann}}\email{gunther@math.washington.edu}

\affil*[1]{\orgdiv{Department of Mathematics}, \orgname{Indian Institute of Science Education and Research}, \city{Bhopal}, \country{India}}

\affil[2]{\orgdiv{College of Arts and Sciences}, \orgname{New York Institute of Technology}, \orgaddress{ \city{New York}, \state{NY}, \country{USA}}}

\affil[3]{\orgdiv{Department of Mathematics}, \orgname{University of Washington}, \orgaddress{ \city{Seattle}, \state{WA}, \country{USA}}}

\abstract{
We revisit the problem of recovering wave speeds and density across a curved interface from reflected wave amplitudes. 
Such amplitudes have been exploited for decades in (exploration) seismology in this context. However, the analysis in seismology has been based on linearization and mostly flat interfaces. Here, we present a nonlinear analysis allowing curved interfaces, establish uniqueness and provide a reconstruction, while making the notion of amplitude precise through a procedure rooted in microlocal analysis. 
}

\maketitle 
\newline

\subsection*{Acknowledgements
} M.V.d.H. gratefully acknowledges support from the Simons Foundation under the MATH + X program, the National Science Foundation under grant
DMS-1815143, and the corporate members of the Geo-Mathematical Imaging Group at Rice University. G.U. was partly supported by NSF, a Walker Family Endowed Professorship at UW and a Si-Yuan Professorship at IAS, HKUST. S.B. was partly supported by Project no.: 16305018 of the Hong Kong Research Grant Council.

\section{Introduction}
We consider the problem of recovering piecewise smooth wave speeds and density across a curved interface from reflected wave amplitudes. 
Such amplitudes have been exploited for decades in seismology in this context. However, the analysis in seismology uses a linearization and assumes mostly flat interfaces. Here, we present a nonlinear analysis allowing curved interfaces, establish uniqueness and provide a reconstruction, while making the notion of amplitude precise through a procedure rooted in microlocal analysis. While our focus is on elastic waves in isotropic media, we consider in parallel the acoustic case with vanishing shear modulus.
By measuring the amplitudes of reflected acoustic or elastic waves above a curved interface at various incidence angles, we recover the jet of the material parameters infinitesimally below the interface as well as the shape operator associated to the interface.

Following the notation in \cite{SUV2019transmission} and \cite{Hansen-CPDEInverse}, let $\Omega \subset \RR^3$ be a smooth, bounded domain and $\Gamma$ a closed, smooth hypersurface splitting $\Omega$ into two subdomains $\Omega_+$ and $\Omega_-$. For the isotropic elastic wave equation, we assume the density of mass $\rho$ and Lam\'{e} parameters $\lambda, \mu$ are smooth up to the surface $\Gamma$ with a possible jump there.
  When an acoustic or elastic wave hits an interface, the strength of the reflected wave depends not only on the material parameters infinitesimally above and below the interface, but it also depends on the angle of incidence and the curvature of the interface. There is a certain reflection operator that we denote $R$ throughout this work, which will be a PsiDO of order zero on the interface, that determines the amplitude of a reflected wave. Concretely, we aim to recover all material parameters (and their derivatives) directly below an interface from knowledge of $R$ and the material parameters above an interface. We will also determine the curvature of the interface from such data. This is a variant of the boundary determination problem (see for example \cite{RachBoundary}) from the hyperbolic Dirichlet-to-Neumann map, but in our case, reflected amplitudes (specifically, we use the full symbol of $R$) is in place of the Dirichlet-to-Neumann map. We are not aware of mathematical literature of this exact problem even though there is plenty of geophysical literature on a simplified version of this problem (see \cite{DavydenkoScatteringandReflection,HammadAVO,Skopintsevacurvatureofinterface}). We treat this inverse problem both in the acoustic and elastic wave setting.

The result closest to ours was obtained by Rachele in \cite{RachBoundary}. She showed that one can uniquely determine the Lam\'{e} parameters and density of mass (including all their derivatives) at the boundary of a domain from the hyperbolic Dirichlet-to-Neumann map.
  Aside from boundary recovery, through asymptotic analysis, following the propagation of singularities, amplitudes and reflection coefficients have been used by seismologists to obtain wave speeds and density of mass just below an interface. The procedure, which is derived from a linearization of the inverse problem considered here, is termed ``amplitude versus offset analysis'' (AVO). The procedure has ``locally'' elastic plane waves sent into an elastic medium with a reflector at various angles and their reflected amplitudes are measured. The amplitude variation due to change in angle of a wave hitting the reflector indicates contrasts in lithology, shear properties, and fluid content in rocks above and below the interface. Recent work for this type of analysis in heterogeneous media can be found in \cite{HammadAVO}, and an inverse problem that incorporates both multiple scattering and recovering reflection coefficients can be found in \cite{DavydenkoScatteringandReflection}. A genetic algorithm for nonlinear recovery of material parameters from reflection coefficients can be found in \cite{deHoopAVO00}. We also refer to \cite{HoopAVO_1997} for an inverse problem with anisotropic media and reflection coefficients. Our analysis here gives a concrete mathematical framework and proof that at least in an isotropic setting (we will study anisotropic settings in another work), one can determine the full elastic properties and density across an interface from reflection operators. Most works in the geophysics literature assumed some type of homogeneity and a simple interface, while several works such as \cite{Skopintsevacurvatureofinterface,Cerveny1974Curvature} consider curved interfaces and their effect on the reflection coefficients. We make no simplifying geometric assumptions about the interface except that it is a smooth hypersurface, and a nice byproduct of our construction, aside from the inverse problem, shows concretely the effect of curvature on the reflection operators. In addition, we do not linearize the problem as is usually done in articles on AVO. Much of our proof is constructive, and will serve as a basis for reconstruction algorithms.

  Our primary motivation is to eventually recover a piecewise smooth density of mass (in the isotropic elastic wave equation) in the interior of the domain, which we present in a subsequent paper. We essentially want to ``image'' the density using high frequency waves. 
  In fact, one may recover piecewise smooth wavespeeds without using reflected amplitudes as in \cite{SUV2019transmission}. In \cite{SUV2019transmission}, Stefanov, Vasy, and Uhlmann first construct the parametrix for the isotropic elastic wave equation away from any glancing rays. They use the principal symbol of the parametrix (in particular, the polarization set) to then recover local travel times for the $P$ and $S$ wavespeeds near a particular interface. By using rays that are near tangential to the interface, they can recover travel times (for both the $P$ and $S$ wavespeeds) between two nearby points at the interface. This allows them to recover the wavespeeds initially at the interface, and then in the interior using local boundary rigidity theorems.
  This argument only relies on the principal symbol of the elastic operator and the parameterix. As noted in their \cite[Remark 10.2]{SUV2019transmission}, their argument does not address unique determination of the density of mass past the first interface, nor at the interface itself. Since the density appears in the lower order part of the elastic operator, it is natural to look at the lower order symbols of the elastic parametrix to initially recover the density at the interface. This leads to the inverse problem considered here where we study the full symbol of the reflection operator, which is a constituent of the elastic parameterix \cite{CHKUElastic}, to recover the jet of the density of mass at the interface.

   In the smooth setting, boundary determination of material parameters is usually needed to prove uniqueness in the interior \cite{Oksanen2020}, while in our setting where material parameters have jump discontinuities, unique determination at the interface is needed to solve the interior problem for the density of mass (recovering piecewise smooth wavespeeds in the interior may already be done without such an interface determination result as in \cite{SUV2019transmission}). The focus of this paper is to prove an interface determination result from reflected amplitudes.
  The proof leads us to compute the full symbol of the acoustic and elastic reflection (and transmission) operators that are used to construct a parametrix to solve the acoustic/elastic wave equation near an interface. We do not know of any literature that has done this computation beyond the principal symbol level, while the principal symbol is computed in many works such as \cite{SUV2019transmission,CHKUElastic,Knott1899,Zoeppritz1919} so that this paper may also be viewed as a generalization of these results. We provide a toy example to illustrate the type of inverse problem we are after.

  Consider a simple half space in $\RR^3$
  with a flat interface $\Gamma$ that separates a layer above denoted $\Omega_-$ and a layer below $\Omega_+$. Suppose that there are two piece-wise constant material parameters $c$ and $\rho$. Let $c_{\pm}$ be $c$ restricted to $\Omega_\pm$ and likewise for $\rho.$ An elastic $P$-wave (say) that hits $\Gamma$ at angle $\theta$ from the normal, with transmitted angle $\theta_t$, the reflection coefficient that determines the reflected wave amplitude is
  \[R(\theta) = \frac{\rho_+c_+\cot(\theta) -\rho_-c_-\cot(\theta)}{\rho_+c_+\cot(\theta) + \rho_-c_-\cot(\theta)}.\]
  Ideally, one would like to reconstruct $c_+ - c_-$ and $\rho_+ - \rho_-$ from $R$ at various angles. Due to the nonlinearities involved, this is difficult even in this simplest of settings. Instead, we are interested in determining $c_+$ and $\rho_+$ from knowing $c_-, \rho_-$ and $R$ for different $\theta$ values. In our setting, we allow $c$ and $\rho$ to be piecewise smooth functions and $\Gamma$ is not restricted to be flat. Hence, we want to determine all derivatives of $c_+$ and $\rho_+$ restricted to $\Gamma$, and the shape operator of $\Gamma$.

 We consider two types of wave fields: acoustic and elastic waves, defined in a bounded domain $\Omega\subset \mathbb{R}^3$.
Though the inverse problems for the acoustic as well as the elastic waves have been studied extensively in the last few decades, a large portion of them concerns domains with smooth material parameters. This manuscript is the analog to a boundary determination result that will enable future results in interior determination, including the density of mass, when material parameters contain conormal singularities. The history of boundary determination problems of a coefficient from the Dirichlet-to-Neumann map is summarized in \cite{RachBoundary}. In the case of the conductivity equation involving an elliptic partial differential operator, Kohn and Vogelius \cite{KV84} and then Sylvester and Uhlmann \cite{SylU88}, prove boundary determination in the case of real analytic and then $C^\infty$ conductivities.  Sylvester and Uhlmann \cite{SylU87} and Nachman \cite{Na96} use the result at the boundary to show an interior uniqueness result in certain situations.

As described in \cite{RachBoundary}, in the case of a scalar hyperbolic wave operator associated with a Laplace-Beltrami operator $\Delta_g$ for a metric $g$, Sylvester and Uhlmann \cite{SylU91} show that the Dirichlet-to-Neumann map uniquely determines the metric (up to the pullback by a diffeomorphism of $\Omega$ that fixes $\p \Omega$) to infinite order at the boundary $\Omega$. In this setting, boundary determination of multiple parameters with stability results can be found in \cite{Montalto2014} and \cite{stefanovYang2018}.
For the elastic setting, Nakamura and Uhlmann have solved the inverse problem for elasticity in the static case in \cite{NU93,NU94} together with the erratum \cite{NU03Erratum}, and independently in \cite{Eskin_2002}, Eskin and Ralston proved uniqueness for both Lam\'{e} coefficients in the isotropic setting when the Lam\'{e} paramater $\mu(x)$ is close to a constant.
  The classic paper for boundary determination of parameters for the isotropic elastic wave equation is \cite{RachBoundary}. An extension of that result to elastic media with residual stress is in \cite{RachResidualStress}, and boundary determination in certain anisotropic cases can be found in \cite{HNZ19}.  Our inverse problem is an analog to these results where the interface acts as our boundary and we instead determine the jet of multiple material parameters at one side of an interface from the reflection operator and knowledge of the material parameters on the other side of the interface.

There is a natural forward problem associated to the inverse problem of this paper. In his 1975 paper \cite{Taylor75}, Michael Taylor microlocally analyzed the reflection and transmission of waves from a boundary or an interface. The scattering due to the boundary was governed by a boundary PsiDO denoted $\beta \in \Psi^0(\p \Omega)$ in \cite{Taylor75} (giving boundary conditions) under certain geometric assumptions such as away from any glancing rays. Transmission conditions can locally be written as a boundary value problem (see \cite{Taylor75,yamamoto1989}) with such a PsiDO as well.
Taylor uses the original operator and the boundary conditions to construct tangential pseudodifferential operators, that he denotes $P^I,P^{II}, P^{III}, P^{IV}$, near the boundary. For a vector valued solution $u$ to a hyperbolic partial differential equation with boundary conditions, $P^I u\restriction_{\Omega}$ roughly represents the trace at the boundary of the ``incoming waves'' and $P^{II}u\restriction_{\Omega}$ represents the trace at the boundary of ``outgoing waves''. The boundary condition leads to a pseodifferential equation involving a derived operator $\gamma$ at the boundary \cite[Equation (3.2)]{Taylor75}. When this equation is elliptic, one may construct a parametrix near the boundary with a constituent at the boundary relating the incoming and outgoing waves. Our inverse problem is: Given the full symbol of several entries of a matrix PsiDO $\gamma$ (or entries of a pseudodifferential operator involving submatrices of $\gamma$) that microlocally determines the amplitudes of scattered waves at an interface, can we recover the jets of certain parameters of the partial differential operator at the boundary, in particular for PDEs describing acoustic and elastic waves? We are not aware of any such results for interfaces.

\subsection{Basic notations and definitions}\label{Sec_Setup}
In this article we consider two types of wave operators, one is the acoustic wave operator $P$, given by (\ref{Acoustic_OP}) defined on a function $u(t,x)$, the second is the linear, isotropic elastic wave operator $Q$, given in (\ref{eq_2}), acting on a vector-field $u(t,x) = (u_1(t,x),u_2(t,x),u_3(t,x))$.
Here we fix our basic assumptions and notational conventions for the rest of the article. Though we work with two types of wave operators, the following discussion in this section is common for both of them. Later we divide this article into two sections, each dedicated to the two types of waves.

Throughout this article, we will work on the space $(0,\infty)\times \R^3$ or its subsets.
We denote $(t,x)$ to be the coordinates on the space $(0,\infty)\times \R^3$.
Let $\Omega \subset \R^3$ be an open bounded domain with smooth boundary.
We assume that the parameters $\mu,\rho$ in the case of the acoustic waves and $\lambda,\mu,\rho$ in the case of the elastic waves are piecewise smooth functions in $\Omega$. Concretely, we assume that $\lambda, \mu, \rho$ are smooth on $\bar \Omega$ except for a jump discontinuity at a smooth closed connected hypersurface $\Gamma \subset \Omega$.
In general, $\Gamma$ can be a collection of disjoint closed connected, orientable hypersurfaces, but we can deal with multiple interfaces via an iterative argument. For the purpose of this article, we restrict ourselves to the fact that $\Gamma$ is a single smooth closed hypersurface in $\Omega$.
 We define $\Omega_{\pm}$ to be the portions of $\Omega$ on the two sides of $\Gamma$, where $\Omega_{-}$ is the portion outside $\Gamma$ and $\Omega_{+}$ is the part inside $\Gamma$.

Let $u_I$ be a wave field (Acoustic/Elastic), travelling through $\Omega_{-}$, approaching the interface $\Gamma$. We write the suffix $I$ to indicate $u_I$ to be an incoming wave field (see \cite{SU-TATBrain, SUV2019transmission} for more details).
After hitting the interface $\Gamma$, the wave field splits into two parts $u_R$ and $u_T$, where $u_R$ is the reflected wave field travelling through $\Omega_{-}$ and
$u_T$ is the transmitted wave field travelling through $\Omega_{+}$, perturbed by a refraction according to Snell's law.
The wave fields $u_{I}$, $u_R$, $u_T$ are standardly related by the transmission conditions corresponding to the acoustic or the elastic waves, given on $(0,\infty)\times\Gamma$.
 Thus, we can write the solution $u$ of the acoustic/elastic wave equation near an interface as (see \cite{SUV2019transmission})
\begin{equation}\label{e: u = uI + uR+uT}
u = u_I + u_R + u_T,
\end{equation}
where $u_I$ is the incident wave, $u_R$ is the reflected wave, and $u_T$ is the transmitted wave, where $u_T$ is supported in $\bar{\Omega}_+$ and $u_I,u_R$ are supported in $\bar\Omega_-$. Using various incident waves, we are interested in whether we can determine all the elastic parameters on $\Gamma^+$ from $u_R$. We restrict ourselves only to hyperbolic ``points'' (see \cite{SUV2019transmission}) and this is sufficient data. In this article we prove that, by knowing the material parameters on one side of the interface along with the knowledge of $u_R$ at the interface, one can determine those parameters and their derivatives (of any order) on $\Gamma$. The theorems are stated precisely for the acoustic case in Section \ref{Sec_Acoustic} and in Section \ref{s: elastic statement of theorem} for the elastic case.

Without loss of generality, we assume $\Gamma \subset \{x_3 = 0\}$ in $\R^3$ is a closed, connected smooth hypersurface. We show in Section \ref{s: nonflat case} how the general case of curved interfaces follows quite easily with some additional terms showing the effect of curvature on the reflection operator. Using local diffeomorphisms, we can set $\Gamma$ to be any closed, connected, smooth hypersurface in $\R^3$, but for the sake of simplicity, at present, we work with $\Gamma \subset \{x_3 = 0\}$. We write the $x_3 < 0 $ is ``above'' while $x_3 > 0$ is ``below'' the interface.
Since our analysis in this article is mostly on the interface $\Gamma$, therefore, we can shrink $\Omega$ to be a small neighborhood of $\Gamma$ in $\R^3$. For notational convenience, we denote $\Gamma_{\pm}$ to be two copies of $\Gamma$ when approached from $\Omega_{\pm}$. We add the suffix $(\pm)$ to denote parameters on the different sides of $\Gamma$. For instance, we write $\rho^{(\pm)}$ to denote the density function $\rho$ on the domains $\Omega_{\pm}$, and similarly for the other material parameters.
 We write $\Omega := \overline{\Omega}_{-} \sqcup \overline{\Omega}_{+}$ to be the disjoint union of $\overline{\Omega}_{+}$ and $\overline{\Omega}_{-}$, with $\Gamma_\pm$ included in the respective boundaries.
For any function $p(x)$ on $\Omega$, we denote $p\restriction_{\Gamma_\mp}$ as the limit of $p(x)$ as $x$ approaches $\Gamma$ from above/below. It will also be convenient to denote $p^{(\mp)} = p\restriction_{\Gamma_\mp}$. We also denote $\p_{\nu}p$ as the normal derivative to $\Gamma$ where $\nu$ is a fixed unit normal to $\Gamma$.

We consider two sets of parameters $(\mu,\rho)$, $(\widetilde{\mu},\widetilde{\rho})$ for the acoustic wave equation and $(\lambda,\mu,\rho)$, $(\widetilde{\lambda}, \widetilde{\mu},\widetilde{\rho})$ for the isotropic elastodynamic wave operator on $\Omega$.
We have two acoustic wave operators $P$ and $\widetilde{P}$, corresponding to the two sets of parameters $(\mu,\rho)$, $(\widetilde{\mu},\widetilde{\rho})$ and elastic wave operators $\Op$, $\widetilde{\Op}$ for the two sets of parameters $(\lambda,\mu,\rho)$, $(\widetilde{\lambda}, \widetilde{\mu},\widetilde{\rho})$ respectively.
 We use the notation $\widetilde{f}$ to refer to a corresponding quantity associated to $\widetilde{P}$(or $\widetilde{\Op}$) when $f$ is a quantity associated to $P$(or $\Op$). In the next two subsections, we state the main theorems in the acoustic and elastic cases.

\subsection{Notation and statement of the theorems}\label{s: elastic statement of theorem}

\subsection*{Acoustic case}
First, consider the acoustic wave equation written in the classical form as
 \begin{equation}\label{e: acoustic wave eq strong form}
 Pu := \rho\p_{t}^2u - \nabla_x\cdot \mu \nabla_{x}u =0,
\end{equation}
  where $u$ is a scalar function, and $\rho(x)$, $\mu(x)$ are two piecewise smooth functions. This is not the standard notation for an acoustic wave equation. Normally, one first considers the elastic equation (\ref{eq_2}), and $\kappa:= \lambda + 2/3\mu$ is the incompressibility (or bulk modulus). In the fluid regions, one sets $\mu \equiv 0$ in (\ref{eq_2}). One then obtains the acoustic wave equation for the pressure field $p = -\kappa \nabla_x \cdot \p_t u$, that is $\kappa^{-1} \p^2_{t} p - \nabla \cdot (\rho^{-1}\nabla p) = 0$. This notation is awkward to use in our paper since we will compare our formulas to that of \cite{RachBoundary, RachDensity}. In order to make the comparison clearer, we replace $\kappa^{-1}$ by $\rho$ and $\rho^{-1}$ by $\mu$ to obtain (\ref{e: acoustic wave eq strong form}), and we will refer to it as the acoustic wave equation. This will allow for easy comparisons between our formulas and to those of \cite{RachBoundary, RachDensity}.

On a bounded domain $\Omega$ with smooth boundary $\partial\Omega$, the wave equation for $P$ with initial Cauchy data at time $t=0$ and a boundary condition on $(0,\infty)\times \partial\Omega$ is well-posed. The acoustic wave equation in an open, bounded domain $\Omega \subset \R^3$ with transmission conditions is given as
\beq\label{Acoustic_OP}
\begin{aligned}
Pu(t,x) := \left(\rho(x)\p_{t}^2 - \nabla_{x}\cdot\mu(x)\nabla_x\right)u(t,x) =& 0,\qquad &&\mbox{in }(0,\infty)\times\Omega\setminus \Gamma,\\
u\restriction_{\Gamma_-} =& u\restriction_{\Gamma_+}, \\
\mu \frac{\p u}{\p \nu}\restriction_{\Gamma_-} =& \mu \frac{\p u}{\p \nu}\restriction_{\Gamma_+},\\
u(t,x) =& f(t,x) \quad &&\mbox{on }(0,\infty)\times\partial\Omega,\\
u\restriction_{t=0} = 0, \quad &\partial_t u\restriction_{t=0} = 0\quad &&\mbox{on }\Omega,
\end{aligned}
\eeq
where as in \cite{SU-TATBrain}, $u\restriction_{\Gamma_\mp}$ is the limit value (the trace) of $u$ on $\Gamma$ when taking the limit from ``above'' and from ``below'' $\Gamma$ respectively. We will denote this limit when applied to a parameter $c$ by $c^{(\pm)}$. We similarly define the interior and exterior normal derivatives, and $\nu$ is the exterior unit (in the Euclidean metric) normal to $\Gamma$. The conditions at $\Gamma_\pm$ are called the \emph{transmission conditions.}
They can be shortened to $[u] = 0$, and $[\mathcal N u]=0$ where $[v]$ stands for the jump of $v$ from the exterior to the interior across $\Gamma$ and $\mathcal N$ is the normal operator: $\mathcal N u = \mu \frac{\p u}{\p \nu}$.
We consider the parameters $\mu$ and $\rho$ to be piecewise smooth in $\Omega$.
The set of discontinuities of $\mu$ and $\rho$ is known as the \emph{interface} and denoted by $\Gamma$.

We denote $c_S:=\sqrt{\mu/\rho}$ to be the wave speed in $\Omega$.
For each given $f$, a solution $u$ to \eqref{Acoustic_OP}, can be written microlocally near the interface in the form (\ref{e: u = uI + uR+uT}), with $u_I$ the incoming wave hitting $\Gamma_-$, $u_R$ the reflected wave, and $u_T$ the transmitted wave initially moving away from $\Gamma_+$ inside $\Omega_+$ (see \cite[Section 4]{SU-TATBrain} for the construction).\footnote{The notion of incoming and outgoing is characterized in terms of its wavefront set. See \cite{SUV2019transmission,SU-TATBrain}.} Denote
\begin{equation}\label{e: definition of h at Gamma}
h := \rho_{\Gamma_-}u_I \in \mathcal{E}'(\Gamma_- \times \R_t),
\end{equation}
 where $\rho_{\Gamma^-}$ is the restriction to $\Gamma$ from $\Omega_{-}$. We assume $h$ is microsupported away from the glancing set (see \cite{SUV2019transmission} for the relevant definition). It is well known that \cite{SU-TATBrain}
 \[\rho_{\Gamma_-}u_R \equiv R(\rho_- u_I) = R h,
 \]
 where $R \in \Psi_{cl}^0(\Gamma_- \times \R_t)$ is the reflection operator (see \cite{CHKUControl,CHKUElastic}) derived explicitly in section \ref{s: acoustic zeroth order derivation} with principal symbol given by $(a_R)_0$ in equation \eqref{zeroth_order_wave}, and `$\equiv$' denotes equality modulo functions in the class $C^\infty(\Gamma_{\pm} \times \R_t)$, and when used between pseudodifferential operators, it means equality modulo operators in $\Psi^{-\infty}$. In our notation, $R$ has principal symbol $b_R^{(0)}\restriction_\Gamma$ in the notation of \cite[Section 4]{SU-TATBrain} and $R$ is constructed microlocally from $P$ and the transmission conditions so that if $\tilde P$ is another acoustic wave operator, there a corresponding reflection operator $\tilde R$. Since we are interested in recovering the material parameters and their derivatives on $\Gamma$, we can shrink $\Omega$ so that $\Omega_{\pm}$ becomes a small one-sided neighbourhoods of $\Gamma_{\pm}$. Hence, without loss of generality, we assume $\Omega$ to be a thin open neighbourhood of $\Gamma$ in $\R^3$.
\begin{theorem}\label{Th_main_acoustic}
Let $P$ and $\widetilde{P}$ be two acoustic wave operators with parameters $(\mu,\rho)$ and $(\widetilde{\mu},\widetilde{\rho})$ on $(0,T_0)\times\Omega$. We assume the notations defined above.
Suppose that $R \equiv \tilde R$ on $(0,T_0)\times\Gamma_-$ and $c_{S} =\tilde c_{S}$ and $\rho = \tilde \rho$ in a small neighbourhood of $\Gamma_{-}$ in $\Omega_-$.
Then $\p_{\nu}^j c^{(+)}_{S} = \p_{\nu}^j\tilde c^{(+)}_{S}$ and $\p_{\nu}^j\rho^{(+)} = \p_{\nu}^j\tilde \rho^{(+)}$ on $\Gamma_+$ for $j = 0, 1, 2, \dots$.
\end{theorem}

\begin{rem}
	Note that we need the data only on a infinitesimally small neighbourhood on only one side of the interface. From the data measured on one side of the interface, we can determine information about the transmitted wave which lies on the other side of the interface. 
\end{rem}

\begin{rem}
	We require measurements for a very short period of time at the interface, but we consider time $T_0>0$ for the wave to travel from $\partial\Omega$ to $\Gamma$. That is, we start with generating an initial pulse at $t=0$ on $\partial\Omega$. Let $T>0$ be the time required for the wave to reach the interface $\Gamma$.
	We consider $T_0>T$ slightly bigger than $T_0$ and take measurements for a small time neighbourhood of $T$ at $\Gamma$. 
\end{rem}

\begin{rem}\label{rem_1}
	An estimate for the time $T_0$ can be given as $\mbox{diam}_{g}(\Omega) < T_0 < 2\mbox{diam}_{g}(\Omega)$, where the diameter is taken in the wave-speed metric $g:=c_S^{-2}dx^2$ in $\Omega$.
	Now, since the wave-speed $c_S$ is piecewise smooth in $\Omega$, we take the distance function defined by adding lengths of the connected geodesics on $\Omega_{-}$ and $\Omega_{+}$. To see in details of such non-smooth distance functions, see \cite{CHKUControl}.
\end{rem}
%

\subsection*{Elastic case}
For a bounded domain $\Omega \subset \R^3$ (representing an elastic object), we consider the isotropic elastic equation with operator $\Op$ given formally in the classical form as
\begin{equation}\label{eq_2}
\Op u:= \rho\p_{t}^2u - \nabla_x \cdot( \lambda \div \otimes \Id + 2\mu \widehat{\nabla}_x)u = 0,
\end{equation}
where $\rho$ is the density, $\lambda$ and $\mu$ are the Lam\'{e} parameters, and $\widehat{\nabla}$ is the \emph{symmetric gradient} used to define the strain tensor for an elastic system via $\widehat{\nabla} u = (\nabla u + (\nabla u)^T)/2$ for a vector valued distribution $u$.
Operator $\Op$ acts on a vector-valued distribution $u(x,t)= (u_1,u_2,u_3)$, the \emph{displacement} of the elastic object.
We assume that the Lam\'{e} parameters $\lambda(x)$ and $\mu(x)$ are bounded and satisfy the \emph{strong convexity conditions}, namely $\mu>0$ and $3\lambda+2\mu >0$ on $\overline{\Omega}$.

We consider the initial boundary value problem as
\begin{align}\label{Elastic_OP_1}
\Op U(t,x) =& 0,\qquad &&\mbox{in }(0,\infty)\times\Omega,\\
U(t,x)\restriction_{\partial\Omega} =& f(t,x), \qquad &&\mbox{for } (t,x) \in (0,\infty)\times\partial\Omega,\\
U\restriction_{t=0} = 0,\quad \p_t U\restriction_{t=0} =& 0, \qquad &&\mbox{in }\Omega.
\end{align}
Here $U(t,x)$ in the above system denotes the displacement in $\Omega$ at time $t\geq 0$ and $c_P = \sqrt{\frac{\lambda+2\mu}{\rho}}$, $c_S = \sqrt{\frac{\mu}{\rho}}$ are the compressional and the shear wave-speeds in $\Omega$ respectively. As described in \cite{SUV2019transmission,CHKUElastic} we impose the following transmission conditions
\begin{equation}\label{e: elastic transmission conditions}
[U] = 0, \qquad [ \mathcal N u ] = 0
\end{equation}
where $[v]$ stands for the jump of $v$ from the exterior to the interior across $\Gamma$ and $\mathcal N f$ are the normal components of the stress tensor (see (\ref{elastic_Neumann})).

The strong convexity conditions on $\lambda$, $\mu$ ensures that $c_P > c_S$ on $\overline{\Omega}$. Near an interface, we have the decomposition $u = u_I + u_R + u_T$ from (\ref{e: u = uI + uR+uT}). As in the acoustic setting, with $h = \rho_{\Gamma^-} u_I$ and assuming $h$ is microsupported away from the glancing set (see \cite{SUV2019transmission} for the relevant definition), one has $\rho_{\Gamma_-}u_R \equiv R h$, where $R \in \Psi^0(\Gamma_- \times \R_t)$ is the reflection operator (see \cite{CHKUControl,CHKUElastic}) and it is denoted by $M_R$ in \cite{CHKUElastic}, `$\equiv$' denotes equality modulo functions in the class $C^\infty(\Gamma_{\pm} \times \R_t)$. The operator $R$ will be derived microlocally from $Q$ and the transmission conditions. Note that in this setting, $\Psi^*(\Gamma_- \times \R_t)$ are pseudodifferential operators operating on vector bundles.
We now state our theorem:
\begin{theorem}\label{t: Elastic case}
	Let $\Op$ and $\widetilde{\Op}$ be two isotropic elastic wave operators with parameters $(\lambda,\mu,\rho)$ and $(\widetilde{\lambda},\widetilde{\mu},\widetilde{\rho})$ on $(0,T_0)\times\Omega$. Assuming the above notational conventions,
	suppose that $R \equiv \tilde R$ on $(0,T_0)\times\Gamma_-$ and $c_{\PS} =\tilde c_{\PS}$, $\rho = \tilde \rho$ near $\Gamma$ in $\overline \Omega_-$.
	Then $\p_{\nu}^j c^{(+)}_{\PS} = \p_{\nu}^j\tilde c^{(+)}_{\PS}$ and $\p_{\nu}^j\rho^{(+)} = \p_{\nu}^j\tilde \rho^{(+)}$ on $\Gamma_+$ for all $j = 0, 1, 2, \dots$.
\end{theorem}

\begin{rem}
 More generally, we can show unique determination of the parameters below the interface using \emph{relative amplitude reflections}. In seismic experiments, one often only has access to normalized reflected amplitudes rather than the exact ones (see Remark \ref{rem: relative amplitudes computation}) for the normalization (c.f. \cite{ZhouRelAmp}). Hence, it becomes natural to ask whether our reconstruction methods would apply to these cases. For a reconstruction formula, the answer is essentially yes, modulo solving an intricate nonlinear equation involving one of the parameters, while unique determination can be done completely. The argument is briefly summarized in Remark \ref{rem: relative amplitudes computation}, which deals with the simpler acoustic case, but similar arguments hold for the elastic case.
\end{rem}

We also state an obvious corollary regarding the unique recovery of the transmission operator $T \in \Psi^0(\Gamma_- \times \RR)$ using the reflection operator. Here, $\rho_{\Gamma_+} u_T = Th$ with $u_T$ and $h$ introduced earlier. On a principal symbol level, this is usually proved using a conservation of energy argument. However, since we show that the full symbol of the transmission operator is determined by the jet of all three parameters on both sides of the interface, we also have
\begin{cor}\label{c: recover T from R}
	Suppose that $R  = \tilde R \text{ mod }\Psi^{-\infty}(\Gamma_- \times \RR)$ and $c_{\PS} =\tilde c_{\PS}$ and $\rho = \tilde \rho$ in $\overline \Omega_-$. Then $T  = \tilde T \text{ mod }\Psi^{-\infty}(\Gamma_{\pm} \times \RR)$.
\end{cor}


\begin{rem}
 We require measurements for a very short period of time at the interface, but we consider a time $T_0>0$ for the waves to travel from $\partial\Omega$ to $\Gamma$. Since, we have two waves (compressional and shear) travelling with different wave-speeds ($c_{\PS}$), the estimate of the time $T_0$ is not as straight-forward as in the acoustic case (Remark \ref{rem_1}).
	We avoid this dilemma with the help of the strong convexity condition of the Lam\'e parameters, which ensures that $c_P>c_S$ in $\Omega$.
	We can define distance functions corresponding to the non-smooth metrics $g_{\PS}$ by joining geodesics on the both sides of the interface (see \cite{CHKUElastic}) and estimate the time $T_0$ for elastic waves as $\mbox{diam}_{g_S}(\Omega) < T_0 < 2\mbox{diam}_{g_P}(\Omega)$, where the two wave-speed metrics are given as $g_{\PS}:=c_{\PS}^{-2}dx^2$ in $\Omega$.
\end{rem}

\begin{rem}
	Since we can take $\Omega$ to be as arbitrarily small and $T_0$ is bounded above and below by $\text{diam}_{g_{\PS}}(\Omega)$, therefore, one can have $T_0$ to be small enough by choosing a thin enough $\Omega$.
\end{rem}

For clarity of the exposition, we first assume a flat interface and prove the theorems in this case in Sections \ref{Sec_Acoustic} and \ref{Sec_statement}. In Section \ref{s: nonflat case}, we extend the arguments to the general case.

\section{Acoustic waves and proof of Theorem \ref{Th_main_acoustic}}\label{Sec_Acoustic}

The parameters for the acoustic waves are $\mu(x)>0$ and $\rho(x)>0$, where $\rho$ is the density of the domain $\Omega$ and $c_S := \sqrt{\frac{\mu}{\rho}}$ is the wave speed in $\Omega$.
We write the coordinates in $(0,\infty)\times\Omega$ as $(t,x)=(t,x_1,x_2,x_3) = (t,x',x_3)$ and consider $(\tau,\xi)=(\tau,\xi_1,\xi_2,\xi_3):=(\tau,\xi',\xi_3)$ to be the dual coordinates of $(t,x)$ in the cotangent space $T^*\Omega$.

\subsubsection*{Summary of the proof of Theorem \ref{Th_main_acoustic}} Let us here provide a brief summary on how we prove Theorem \ref{Th_main_acoustic}.
As a first step, we give a complete derivation of the reflection operator $R$ by using ``geometric optics solutions'' of \eqref{Acoustic_OP} near $\Gamma$. One can use Fourier integral operators to construct a microlocal parameterix for the fundamental solution of \eqref{Acoustic_OP} when $f$ is a delta distribution. After imposing the transmission conditions, we derive the reflection operator $R$ that is a constituent of this parameterix. $R$ will end up being a $0$'th order classical PsiDO  and we derive each symbol in the polyhomogeneous expansion of the symbol of $R$. We also show how the curvature of $\Gamma$ affects the lower order symbols. 

Afterwards, we proceed with a series of lemmas and propositions showing how to recover the material parameters $\mu$ and $\rho$, and their derivatives, at the interface using each term in the polyhomogeneous expansion of the full symbol of $R$. We start with the principal symbol of $R$ and then successively use lower order symbols to recover more derivatives of the coefficients. The shape operator of $\Gamma$ gets recovered as well. The final proof just combines the lemmas and propositions and will follow easily. 
Since the elastic case follows an analogous procedure in a more complicated case, we leave most of the proofs to Appendix \ref{s: proofs of acoustic lemmas}, and instead focus on the proofs for the elastic case.

We consider a geometric optic solution for the acoustic wave equation \eqref{Acoustic_OP} locally near $\Gamma$ as \[
U = U_I+ U_R + U_T
\]
with
\begin{equation*}
U_{\bullet}(t,x)
= \int e^{i\phi_{\bullet}(t, x, \tau, \xi')} a_{\bullet}(t, x, \tau, \xi')\hat h(\tau, \xi') d \tau d\xi',
\end{equation*}
where $\bullet = I/R/T$ denotes the incoming, reflected or the transmitted wave fields and $\hat{h}$ is the Fourier transform of the $h:=\rho_{\Gamma_{-}}U_I$.
The wave fields $U_{I/R}$ are supported on $\overline\Omega_{-}$ and $U_T$ is supported in $\overline \Omega_{+}$.
The phase function $\phi_{\bullet}(t,x,\tau,\xi')$ satisfies the usual Eikonal equation
\begin{equation}\label{Eikonal_eq}
\abs{\partial_t \phi_{\bullet}}^2 = c^2_{S}\abs{\nabla_{x} \phi_{\bullet}}^2,
\end{equation}
with the boundary condition $\phi_{\bullet}\restriction_{x_3 = 0} = -t\tau + x' \cdot \xi'$. We observe that $\phi_{\bullet}$ is of homogeneity $1$ in the $(\tau,\xi')$ variables.
We write a formal asymptotic series for the amplitude function $a_{\bullet}(t,x,\tau,\xi')$ as
\begin{equation*}
a_{\bullet}(t,x,\tau,\xi^{\prime}) = \sum_{J=0}^{-\infty} (a_{\bullet})_J(t,x,\tau,\xi^{\prime}), \qquad \bullet = I,R,T,
\end{equation*}
where $(a_{\bullet})_J$ is homogeneous of order $|J|$ in $\abs{\left(\tau,\xi^{\prime}\right)}$.
From the equation $PU=0$, separating orders of $\abs{(\tau,\xi^{\prime})}$ we obtain recursive transport equations for the terms $(a_{\bullet})_J$.

Without loss of generality, we assume a flat metric near $\Gamma$, i.e. $g = c_S^{-2}dx^2$ and assume $\Gamma \subset \{ x_3 = 0 \}$ so that $x_3$ is a defining function for $\Gamma$. We show in Section \ref{s: nonflat case} how the general case follows easily from this case. 
The reflection and the transmission operators $R$, $T$ on the interface are derived from the transmission conditions so that $R(U_I) = U_R$ and $T(U_I) = U_T$ when restricted to $\Gamma$. One can calculate the full symbol of $R$ and $T$ microlocally (see \cite{SU-TATBrain,Hansen-CPDEInverse} as well) from the transmission conditions on $\Gamma$ induced by the acoustic wave equation.
In Theorem \ref{Th_main_acoustic} we prove that one can determine $\partial_{\nu}^k\mu$, $\partial_{\nu}^k\rho$ on $\Gamma$, for $k=0,1,2,\dots$, from the knowledge of the reflection operator $R$ at the interface $\Gamma$ and the parameters $\rho$, $\mu$ on $\Omega_{-}$.

\subsection{Derivation of reflection operator}\label{s: acoustic zeroth order derivation}
Since $h(t,x')$ on $\Gamma$ can be made arbitrary, one can work with the acoustic wave parametrix
\begin{equation*}
u_{\bullet}(t,x,\tau,\xi') = e^{i\phi_{\bullet}(t,x,\tau,\xi')}a_{\bullet}(t,x,\tau,\xi').
\end{equation*}
Let $u_I$ be an incoming wave approaching the interface $\Gamma$ and $u_R$ be the reflected wave with the condition $\rho_{\Gamma_{-}}u_R \equiv R h$, where $R \in \Psi^0(\Gamma_{-} \times \RR)$ is a well-known pseudodifferential reflection operator.  Thus, $a_R\restriction_{\Gamma}$ is the symbol of $R$ in the statement of Theorem \ref{Th_main_acoustic} and the discussion preceding it, and the symbol of $R$ has an asymptotic expansion as $\sum_{J=0}^{-\infty} (a_{R})_J\restriction_{\Gamma_-}$.
The interface condition for acoustic waves reads
\begin{align*}\label{e: interface conditions}
a_I + a_R &= a_T, \\
\mu^{(-)}(\p_{x_3} \phi_I a_I + \p_{x_3}a_I) + \mu^{(-)}(\p_{x_3} \phi_R a_R + \p_{x_3}a_R) &= \mu^{(+)}(\p_{x_3} \phi_T a_T + \p_{x_3}a_T ), \\
& \text{ on }\Gamma.
\end{align*}

Now, we must have $u_I\restriction_\Gamma = h$ so that this imposes the boundary conditions of $(a_{I})_J$, $J=0,-1,\dots$. Indeed we get
\begin{equation}\label{e: bdy conditions of a_I}
(a_I)_0 = 1 \text{ and }
(a_I)_J = 0 \text{ on } \Gamma.
\end{equation}
Now, observe that, from the interface conditions of $\phi_{\bullet}$ in (\ref{Eikonal_eq}) we obtain $\p_{x_k}e^{i\phi_{\bullet}} = i\xi_k e^{i\phi_{\bullet}}$ for $k=1,2$ and $\p_{t}e^{i\phi_{\bullet}} = -i\tau e^{i\phi_{\bullet}}$ on $\Gamma$.
Furthermore, we define the quantity
\begin{equation}
\xi_{3, \bullet} = \sqrt{\abs{\p_{x'}\phi_{\bullet}}^2 - c^{-2}_{S} \abs{\p_t \phi_{\bullet}}^2} = \sqrt{\abs{\xi'}^2 - c^{-2}_{S}\abs{\tau}^2}
,\qquad \mbox{where}\quad \bullet = I, R, T.
\end{equation}
Also, note that $\xi_{3,R} = -\xi_{3,I}$ on $\Gamma$. The interface conditions for the $0$'th order term $(a_{\bullet})_0$ are
\begin{align*}
&\begin{cases}
-(a_R)_0 + (a_T)_0 = (a_I)_0 \\
-\mu^{(-)}\xi_{3, R} (a_R)_0 + \mu^{(+)}\xi_{3, T} (a_T)_0
 = \mu^{(-)}\xi_{3, I}((a_I)_0 = 1)
\end{cases}
\qquad \mbox{on }\Gamma.
\\
\text{Hence, } \quad
 &\bmat -1 & 1 \\ -\mu^{(-)}\xi_{3, R} & \mu^{(+)}\xi_{3, T}\emat \col{(a_R)_0 \\ (a_T)_0 }
 = \col{ 1 \\ \mu^{(+)}\xi_{3, I}} \qquad \text{ on } \Gamma.
\end{align*}
Since $\xi_{3, R} = - \xi_{3, I}$, we compute
\begin{equation}\label{zeroth_order_wave}
\col{(a_R)_0\\ (a_T)_0}
= \frac{-1}{\mu^{(-)}\xi_{3, I} + \mu^{(+)}\xi_{3, T}}
\bmat \mu^{(+)}\xi_{3, T} & -1 \\ -\mu^{(-)}\xi_{3,I} & -1\emat
\col{1 \\ \mu^{(+)}\xi_{3, I} } = B_0\col{1 \\ \mu^{(+)}\xi_{3, I} },
\end{equation}
where the matrix $B_0$ above depends only on the parameters $\lambda,\mu,\rho$ at the boundary, but \emph{not} on their derivatives.
For the order of homogeneity $-J = 1,2,\dots$ in $\abs{\xi^{\prime}}$, using the boundary conditions for $(a_{\bullet})_{J}$, we get
\begin{multline}\label{e: transmission conditions}
\bmat -1 & 1 \\ -\mu^{(-)}\xi_{3, R} & \mu^{(+)}\xi_{3, T}\emat\col{(a_R)_J \\ (a_T)_J }
\\ = \col{ 0 \\ \mu^{(-)}\p_{x_3}(a_I)_{J+1} + \mu^{(-)}\p_{x_3}(a_R)_{J+1}-\mu^{(+)}\p_{x_3}(a_T)_{J+1}} \quad \text{on } \Gamma.
\end{multline}

Thus, $(a_R)_0$ restricted to $\RR_t \times \Gamma$ is the principal symbol of $R$ and $(a_R)_J$ for $J = -1, -2, \dots$ restricted to $\RR_t \times \Gamma$ are the lower order symbols in the polyhomogeneous expansion of the symbol of $R$.

\subsection{Some lemmas and proof of Theorem \ref{Th_main_acoustic}}\label{s: lemmas and proof of acoustic thm}

We can uniquely determine both material parameters restricted to the interface from $(a_R)_0$. The proofs of the lemmas in this section are in Appendix \ref{s: proofs of acoustic lemmas}.

\begin{lemma}\label{l: recover 2 params from 0'th order reflect}
Suppose $(a_R)_0(x,\tau_i,\xi'_i) = (\tilde a_R)_0(x,\tau_i,\xi_i')$ for $i=1,2$ such that $\abs{\xi_1'}/\tau_1 \neq \pm \abs{\xi_2'}/\tau_2$ and $(x,\tau_i,\xi'_i)$ are not in the glancing set. Suppose also that $\mu^{(-)} = \tilde \mu^{(-)}, \rho^{(-)} = \tilde \rho^{(-)}$. Then $\mu^{(+)} = \tilde \mu^{(+)}, \rho^{(+)} = \tilde \rho^{(+)}$. That is, the reflection coefficient at two different covectors in the nonglancing region uniquely determine both material parameters infinitesimally below the interface when those parameters are known above the interface.
\end{lemma}

The next step is to recover all the higher order normal derivatives of the parameters at the non-glancing region of the interface from the lower order symbols in the polyhomogeneous expansion of the full symbol of $R$.
Since we are considering the non-glancing region only, therefore, we may very well assume that $\xi_{3,\bullet}$ is bounded away from $0$.

{\bf Notation:} We denote by $R_j$ terms that depend on
\begin{enumerate}
	\item[$\bullet$]
	normal derivatives of $c_S,\rho$ of order at most $j$, and
	\item[$\bullet$]
	quantities determined completely by the transmission conditions (\ref{zeroth_order_wave}),(\ref{e: transmission conditions}) in $\Gamma$ for $J=0,-1, \dots, 1-j$, and
\item[$\bullet$] any quantity in the known region $\overline\Omega_-$.
\end{enumerate}

By a direct calculation, $P u_{\bullet} =0$ reduces to the equation
\begin{equation*}
p(x, \p_{t,x} \phi_\bullet) a_{\bullet}
+ 2i(\rho \p_t \phi_\bullet \p_t - \mu \p_x \phi_\bullet \cdot \p_x)a_\bullet
+ i (P\phi_\bullet)a_\bullet + P(x,D_t,D_x) a_{\bullet} = 0,
\end{equation*}
where $p(t,x,\tau\,\xi)$ is the principal symbol of the operator $P = \rho \p_t^2 - \nabla \cdot \mu \nabla$, given as
\begin{equation*}
p(t,x,\tau,\xi) = -\left(\rho(x)\tau^2 - \mu(x)\abs{\xi}^2\right).
\end{equation*}
Separating orders of $\abs{\xi^{\prime}}$ we obtain the transport equations
\begin{equation}\label{e: transport equations}
\begin{gathered}
(\p_t \phi_\bullet \p_t - c_S^2 \p_x \phi_\bullet \cdot \p_x)(a_\bullet)_0
+  ((1/2\rho)P\phi_\bullet)(a_\bullet)_0 = 0, \\
(\p_t \phi_\bullet \p_t - c_S^2 \p_x \phi_\bullet \cdot \p_x)(a_\bullet)_J
+  ((1/2\rho)P\phi_\bullet)(a_\bullet)_J \\ = -(1/i2\rho)P(x,D_t,D_x) (a_{\bullet})_{J+1},
\quad \mbox{for }J<0.
\end{gathered}
\end{equation}
The Hamiltonian to describe downgoing and upgoing waves is
\begin{equation*}
q_\pm(t,x,\tau,\xi) = \xi_3 \mp \sqrt{c_S^{-2}\tau^2 - \abs{\xi'}^2},
\end{equation*}
with the Hamilton's equations
\begin{align*}
\frac{dt}{ds} = \frac{-\tau}{c_S^2 \xi_{3,\bullet}}
\qquad\mbox{and}\qquad
\frac{dx}{ds} = \frac{\xi}{\xi_{3,\bullet}}.
\end{align*}
Here $s$ is the parameter along the Hamiltonian fields. Along the Hamilton vector fields we obtain
\begin{multline}\label{key_5}
(\p_t \phi_\bullet \p_t - c_S^2 \p_x \phi_\bullet \cdot \p_x)a_\bullet
= -c_S^2\xi_{3,\bullet}  \left(\frac{d}{ds} a_\bullet\right)
\\
= -c_S^2\xi_{3,\bullet} \left(\p_{x_3}a_\bullet
+ (\tau/\xi_{3,\bullet}) \p_t a_\bullet - (\xi'/\xi_{3,\bullet}) \cdot \p_{x'}a_\bullet\right).
\end{multline}
Observe that, when we restrict to $\Gamma$, the second two terms
viz. $(\tau c^2_S)\partial_t a_{\bullet}$ and $c^2_S \left(\xi^{\prime}\cdot\partial_{x^{\prime}}a_{\bullet}\right)$
are completely determined by $0$ derivatives of $c_S$, $\rho$ at $\Gamma$ and the transmission conditions (\ref{e: transmission conditions}).
Thus, using our notation $R_j$, for $J=0$, (\ref{key_5}) can be written as
\begin{equation}\label{e: Hamilton derivative to normal deriv}
(\p_t \phi_\bullet \p_t - c_S^2 \p_x \phi_\bullet \cdot \p_x)(a_\bullet)_0
= -(c_S^2\xi_{3,\bullet})  \frac{d}{ds}(a_\bullet)_0
= -c_S^2\xi_{3,\bullet} \p_{x_3}(a_\bullet)_0
+ R_0.
\end{equation}

Along with the transport equations, this identity shows that the term $\partial_{x_3}(a_\bullet)_0$ can be expressed in terms of the normal derivatives of the parameters at $\Gamma$.
In order to illustrate this fact and to get an explicit relation between the normal derivatives of the amplitude and the normal derivatives of the coefficients, we state the following technical lemma.
%
\begin{lemma}\label{l: acoustic partial_x_3 a_0 formula}
$\p_{x_3}(a_\bullet)_0$ are $R_1$, that is, they depend on at most $1$ derivative of $\rho$, $c_S$ on $\Gamma$.
In fact we have the following explicit relation
\begin{equation}\label{e: acoustic partial_{x_3}a_0 formula}
\p_{x_3}(a_\bullet)_0 = -\left[(\p_{x_3}\log \sqrt \rho)
- (\p_{x_3}\log c_S)\left(1 - \frac{(\p_t \phi_\bullet)^2}{2c_S^2 \xi^2_{3,\bullet}}\right)\right](a_\bullet)_0 + R_0.
\end{equation}

\end{lemma}

Next, higher order normal derivatives of the material parameters can be uniquely determined from the knowledge of $(a_{\bullet})_J$ at $\Gamma$.
We first consider the case for the first order normal derivatives.
\begin{lemma}\label{l: first derivative}
One may recover the first normal-derivatives of both parameters, that is $\p_{x_3}\rho^{(+)}$ and $\p_{x_3}c_S^{(+)}$ at $\Gamma$ from $(a_R)_{-1}$.
\end{lemma}

For the higher order derivatives of the coefficients we have an analogous lemma. For the smooth case, the following lemma reduces to \cite[Lemma 3.10]{RachBoundary}.
\begin{lemma}\label{l: acoustic higher derivatives derivation}
 Fix $J \in \mathbb \{-2, -3, \dots\}$ and assume that $\mu$, $\rho$ are known on $\Omega_{-}$, and $(a_R)_0,\dots,(a_R)_{1+J}$ are known on $\Gamma$. Then
$\p_{x_3}^{\abs{J}}c^{(+)}_S$ and $\p_{x_3}^{\abs{J}}\rho^{(+)}$ are uniquely determined by $(a_R)_J$ at $\Gamma$. In fact, we have the following explicit relation
\begin{align}\label{e: (a_R)_J equation for acoustic}
(a_R)_J &=  -(-i/(2\xi_{3,T}))^J \left[(\p^{\abs{J}}_{x_3}\log \sqrt {\rho^{(+)}}) \right.\\ \nonumber
&\qquad \qquad \qquad
 \left.+\p^{\abs{J}}_{x_3}\log c^{(+)}_S\left(1 - \frac{(\p_t \phi_T)^2}{2c_S^2 \xi^2_{3,T}}\right)\right]\frac{(a_T)_{J+1}}{R_{\abs{J+1}}} + R_{\abs{J+1}}
\end{align}
\end{lemma}

\begin{proof}[Proof of Theorem \ref{Th_main_acoustic}] The assumption $R \equiv \tilde R$ implies that the full symbol of $R$ coincides with the full symbol of $\tilde R$ away from the glancing set. By construction, this means $(a_R)_J = (\tilde a_R)_J$ for each $J$ on $\RR_t \times \Gamma_-$ so by Lemma \ref{l: acoustic higher derivatives derivation}, we may recover $\mu^{(+)}$ and $\rho^{(+)}$. By Lemma \ref{l: acoustic partial_x_3 a_0 formula} and \ref{l: acoustic higher derivatives derivation} we recover $\p_{x_3}^{J}c^{(+)}_S$ and $\p_{x_3}^{J}\rho^{(+)}$ for each $J=1, 2,\dots$.
\end{proof}

\section{Elastic waves and proof of Theorem \ref{t: Elastic case}}\label{Sec_statement}

Recall the isotropic elastodynamic wave equation as
\beq
\begin{aligned}\label{Elastic_OP}
 \rho\p_{t}^2U - \nabla \cdot( \lambda \div \otimes \Id + 2\mu \widehat{\nabla})U =& 0,\qquad &&\mbox{in }(0,\infty)\times\Omega,\\
U(t,x)\restriction_{\partial\Omega} =& f(t,x), \qquad &&\mbox{for } (t,x) \in (0,\infty)\times\partial\Omega,\\
U\restriction_{t=0} = 0,\quad \p_t U\restriction_{t=0} =& 0, \qquad &&\mbox{in }\Omega.
\end{aligned}
\eeq
We define the compressional wave speed $c_P$ and the shear wave speed $c_S$ as
\begin{equation*}
c_P=  \sqrt{\frac{\lambda+2\mu}{\rho}}, \qquad c_S = \sqrt{\frac{\mu}{\rho}}, \quad \mbox{in }\Omega.
\end{equation*}
The proof of Theorem \ref{t: Elastic case} will proceed in a series of steps analogous to the acoustic case and detailed at the start of section \ref{Sec_Acoustic}. As in that case, we start with a geometric optics solution of \eqref{Elastic_OP} near $\Gamma$ and then derive the reflection operator $R$ with its full symbol.

Since we consider our analysis only on $\Gamma$, we can shrink $\Omega$ to be a small neighbourhood of $\Gamma$. Considering $\Omega$ as a neighbourhood of $\Gamma$
we construct the geometric optic solutions for the elastic wave equation \eqref{Elastic_OP} given as
\begin{equation}\label{GO_form}
\left(U_{\bullet}\right)_l
= \sum_{\star=\PS}\sum_{m=1,2,3}\int e^{i\phi_{\bullet,\star}(t, x, \tau, \xi')} A^{l,m}_{\bullet,\star}(t, x, \tau, \xi')\widehat{f}_m(\tau, \xi') d \tau d\xi', \quad l=1,2,3
\end{equation}
where $\bullet = I/R/T$ denotes the incoming, the reflected or the transmitted wave field.
Note that $U_{I/R}$ travels through $\Omega_{-}$, whereas $U_T$ is on $\Omega_{+}$.
The phase functions $\phi_{\bullet,\PS}$ satisfies the Eikonal equations
\begin{equation}\label{elastic_eikonal_equation}
\abs{\p_t \phi_{\bullet,\PS}}^2 = c^2_{\PS}\abs{\nabla_{x} \phi_{\bullet,\PS}}^2,
\quad \mbox{such that }\quad
\phi_{\bullet,\PS}(t,x)\restriction_{x_3=0} = -t \tau + x' \cdot \xi'.
\end{equation}
Similar to the acoustic wave case, $\phi_{\bullet,\PS}$ is homogeneous of order $1$ in $\abs{(\tau,\xi')}$.
We define the quantity
\begin{equation}\label{xi_3_ps}
\xi_{3,\bullet,\PS} := \sqrt{\abs{\nabla_{x'}\phi_{\bullet,\PS}}^2- c_{\PS}^{-2}\abs{\p_t \phi_{\bullet,\PS}}^2},
\end{equation}
and observe that $\xi_{3,R,\PS} = -\xi_{3,I,\PS}$.
The amplitudes $(A_{\bullet, \PS})_J$ are homogeneous of order $|J|$ in $\abs{(\tau,\xi')}$ and solves the following iterative equations
\begin{multline}\label{eq_transport}
p(t,x,\partial_t\phi_{\bullet,P/S},\nabla_x \phi_{\bullet,P/S})(A_{\bullet,\PS})_{J-1}\\
= \mathcal{B}_{\bullet,P/S}(A_{\bullet,\PS})_{J} + \mathcal{C}_{\bullet,P/S}(A_{\bullet,\PS})_{J+1}, \quad \forall J=0,-1,-2,\dots,
\end{multline}
where $(A_{\bullet,\PS})_{1} = 0$ and $p(t,x,\tau,\xi)$ is the principal symbol of the operator $P$, given as $p(t,x,\tau,\xi):= (-\rho\tau^2 + \mu\abs{\xi}^2)\Id + (\lambda+\mu)(\xi\otimes\xi)$. Also, $p_{i,j}$ refers to the $ij$'th entry of the matrix symbol.
The operators $\mathcal{B}_{\bullet,P/S}$ and $\mathcal{C}_{\bullet,P/S}$ are given as
\begin{align*}
\left(\mathcal{B}_{\bullet,P/S}M\right)_{k_1,k_2}
&:=i\left(\partial_{\tau,\xi}p(t,x,\partial_t\phi_{\bullet,P/S},\nabla_x \phi_{\bullet,P/S}) \cdot \partial_{t,x}M\right)_{k_1,k_2}\\
&\qquad - \left(p_1(t,x,\partial_t\phi_{\bullet,P/S},\nabla_x \phi_{\bullet,P/S})M\right)_{k_1,k_2}\\
&\qquad +\frac{i}{2}\sum_{\abs{\alpha}=2}\sum_{l=1}^{3}\partial^{\alpha}_{\tau,\xi}p_{k_1,l}(t,x,\partial_t\phi_{\bullet,P/S},\nabla_x \phi_{\bullet,P/S}) \\
&\qquad \qquad \qquad \qquad \qquad  \cdot \left(\partial^{\alpha}_{t,x}\phi_{\bullet,P/S}\right)M_{l,k_2}\\
\left(\mathcal{C}_{\bullet,P/S}M\right)_{k_1,k_2}
:=\     &i\left(\partial_{\tau,\xi}p_1(t,x,\partial_t\phi_{\bullet,P/S},\nabla_x \phi_{\bullet,P/S}) \cdot \partial_{t,x}M\right)_{k_1,k_2}\\
&+\frac{1}{2}\sum_{\abs{\alpha}=2}\sum_{l=1}^{3}\partial^{\alpha}_{\tau,\xi}p_{k_1,l}(t,x,\partial_t\phi_{\bullet,P/S},\nabla_x \phi_{\bullet,P/S}) \cdot \left(\partial^{\alpha}_{t,x}M_{l,k_2}\right),
\end{align*}
where $p_1(t,x,\tau,\xi) = -i\left[ \nabla_x\lambda \otimes \xi + (\nabla_x\mu\cdot\xi)\Id + (\xi \otimes \nabla_x\mu) \right]$ is the lower order terms of the symbol of $P$.

We consider the elastic wave parametrix as
\begin{multline*}
\left(u_\bullet(t,x,\tau,\xi')\right)_m := \sum_{\star=\PS} e^{i\phi_{\bullet,\star}(t, x, \tau, \xi')} A^{\cdot,m}_{\bullet,\star}(t, x, \tau, \xi')
\\= \sum_{\star=\PS} e^{i\phi_{\bullet,\star}(t, x, \tau, \xi')} a^{m}_{\bullet,\star}(t, x, \tau, \xi'),
\end{multline*}
where $a^{m}_{\bullet,\PS}$ is the vector-field with components to be $\left(a^{m}_{\bullet,\PS}\right)_l = A^{l,m}_{\bullet,\PS}$.
For the sake of notational simplicity, we denote $a^{1}_{\bullet,\PS}$ by $a_{\bullet,\PS}$.

 We define $N:=\frac{\nabla_x\phi_{\bullet,P}} {\abs{\nabla_x\phi_{\bullet,P}}}$ be the unit vector in the kernel of $p(t,x,\partial_t\phi_{\bullet,P},\nabla_x \phi_{\bullet,P})$. Take $N_1,N_2$ two orthonormal vectors in the kernel of $p(t,x,\partial_t\phi_{\bullet,S},\nabla_x \phi_{\bullet,S})$ such that $\{ N_1, N_2, \frac{\nabla_x\phi_{\bullet,S}}{\abs{\nabla_x\phi_{\bullet,s}}}\}$ forms an orthonormal basis for $\R^3$.
From the transport equation \eqref{eq_transport} one easily obtains the following compatibility condition
\begin{equation}\label{eq_compatibility}
N_{P/S}\left[ \mathcal{B}_{\bullet,P/S}(a_{\bullet,\PS})_{J} + \mathcal{C}_{\bullet,P/S}(a_{\bullet,\PS})_{J+1} \right] = 0, \quad \forall J=0,-1,-2,\dots,
\end{equation}
where $N_P = N = \frac{\nabla_x\phi_{\bullet,P}}{\abs{\nabla_x\phi_{\bullet,P}}}$ and $N_S = N_1$ or $N_2$.
The amplitudes $(a_{\bullet,\PS})_0$ are written in the form
\begin{multline}\label{eq_1}
(a_{\bullet,\PS})_J = (h_{\bullet,\PS})_J+\begin{cases}
(\alpha_{0,\bullet})_J N_\bullet & \text{ if }\phi_\bullet = \phi_{\bullet, P} \\
(\alpha_{1,\bullet})_J N_{1, \bullet}
+(\alpha_{2,\bullet})_J N_{2, \bullet} & \text{ if }\phi_\bullet = \phi_{\bullet, S}
\end{cases}, \\ \qquad J=0, -1,\dots,
\end{multline}
for some vector $(h_{\bullet,\PS})_J$ in the co-kernel of $p(t,x,\p_{t,x}\phi_\bullet)$, to be determined  ($(h_{\bullet,\PS})_0 = 0$).
%
%
We write
\begin{equation*}
\mbox{for } J=-1,-2,\dots,\qquad
\begin{cases}
h_{\bullet,P} = (\gamma_{1,\bullet})_J M_1 + (\gamma_{2,\bullet})_J M_2,\\
h_{\bullet,S} = (\gamma_{\bullet})_J M,
\end{cases}
\end{equation*}
where $M$, $M_1$ and $M_2$ are orthogonal to $\PS$ waves, given as
\begin{equation*}
M_1 = -ie_3 \times (\xi',0),
\qquad M_2 = -\xi_{3,\bullet,P}(\xi',0) + \abs{\xi'}^2e_3,
\qquad M = -i\xi_{\bullet,S},
\end{equation*}
where $\xi_{\bullet,\PS} := (\xi',0) + \xi_{3,\bullet,\PS}e_3$ and the $\alpha_{i,\bullet}$ are vector symbols for $i=0,1,2$.

\subsection{P/S mode projections}
First we construct a $\PS$-mode projector $\Pi_{\PS}$, microlocally projects the elastic wave field $u$ to the compressive ($P$) and the shear ($S$) wave fields for a small time-interval, as
$\Pi_{\PS} u_{\bullet} = u_{\bullet,\PS} = e^{i\phi_{\bullet}} a_{\bullet,\PS}$.
Observe that the elasticity operator $\Op$, as defined in \eqref{Elastic_OP}, has the principal symbol $p(t,x,\tau,\xi)$ given by a $3\times 3$-matrix as
\begin{equation}
p(t,x,\tau,\xi) = -\rho\left[\left(\tau^2-c_S^2\abs{\xi}^2\right)I_{3\times3} - \left(c_P^2-c_S^2\right)(\xi\otimes\xi)\right].
\end{equation}

Observe that $p(t,x,\tau,\xi)$ has eigenvalues $\rho\left(\tau^2-c_P^2\abs{\xi}^2\right)$ and $\rho\left(\tau^2-c_S^2\abs{\xi}^2\right)$ with multiplicity $1$ and $2$ respectively. The matrix $p$ can be diagonalised and there exists unitary matrix $V(t,x,\tau,\xi)$ such that
\begin{equation*}
V p(t,x,\tau,\xi) V^{-1} = \rho\begin{pmatrix}\tau^2-c_P^2\abs{\xi}^2 &0 &0\\0 &\tau^2-c_S^2\abs{\xi}^2 &0\\0 &0 &\tau^2-c_S^2\abs{\xi}^2\end{pmatrix} = D(t,x,\tau,\xi).
\end{equation*}
We now consider the symbol
\begin{equation}\label{Symbol_mode-projector}
\Pi_P(t,x,\tau,\xi) := V^{-1}\bmat1 &0 &0\\0 &0 &0\\0 &0 &0\emat V
\qquad \mbox{and}\qquad
\Pi_S(t,x,\tau,\xi) := V^{-1}\bmat 0 &0 &0\\0 &1 &0\\0 &0 &1 \emat V.
\end{equation}
One can equivalently write the mode projection operators $\Pi_{\PS}$ as
\begin{equation*}
\Pi_{\PS}u_{\bullet} = \int e^{i\phi_{\bullet,\PS}}a_{P/S}(t,x,\tau,\xi) \widehat{h}(\tau,\xi')\,d\tau\,d\xi' = u_{\bullet,\PS},
\end{equation*}
where $u_{\bullet}$ is as defined in \eqref{GO_form}.
Observe that the symbol of $\Pi_{\PS}$ is homogeneous of order $0$ in $\abs{\xi}$ and thus $\Pi_{\PS}$ represents a $0$-th order pseudodifferential operator.

\subsection{The elastic transmission conditions}
Let $u_I$ be the incoming wave parametrix corresponding to \eqref{Elastic_OP}, travels through $\Omega_{-}$, approaching the interface $\Gamma$. Let $h:=\rho_{\Gamma_{-}}u_I$, where $\rho_{\Gamma_{\pm}}$ are the restriction operators on $\Gamma_{\pm}$.
Denote $R$ and $T$ to be the well-known reflection and transmission operators on $\Gamma$ respectively. As calculated in \cite{CHKUElastic} $R$, $T$ are pseudo-differential operators ($\Psi$DO) of order $0$ on $\Gamma$.
Let $f := \rho_{\Gamma^-}u_I \in \mathcal{E}'(\Gamma^- \times \RR_t)$, where $\rho_{\Gamma^-}$ is the restriction to $\Gamma$ from above. The reflected wave field $u_R$ and the transmitted wave field $u_T$ starts from $\Gamma_{-}$ and $\Gamma_{+}$ respectively, with the boundary data as $\rho_{\Gamma_{-}} u_R = Rf$ and $\rho_{\Gamma_{+}} u_T = Tf$.


We define the Neumann operator at $\Gamma$, given as
\begin{equation}\label{elastic_Neumann}
\mathcal N_{\bullet} u_{\bullet} = (\lambda \div \otimes \text{I} + 2\mu \hat \nabla)u_{\bullet} \cdot \nu_{\bullet}\restriction_\Gamma,
\end{equation}
where $\nu$ is the outward unit normal vector at $\Gamma$ i.e. for $\bullet = I/R$ consider $\nu_{\bullet}$ to be the normal unit vector on $\Gamma$ pointing towards $\Omega_{+}$ and for $\bullet=T$, $\nu_{T}$ is the unit normal vector pointing towards $\Omega_{-}$.
The elastic transmission conditions on the interface $\Gamma$ from \eqref{e: elastic transmission conditions} become
\begin{align}\label{elastic_transmission}
u_I + u_R =& u_T\\
\mathcal{N}_{I}u_I + \mathcal{N}_{R}u_R =& \mathcal{N}_{T}u_T. \nonumber
\end{align}
Recall that we assume $\Gamma = \{ x_3=0\}$ and $\Omega_{\pm} \subset \{ x\in\R^3 : \pm x_3>0 \}$.
Now, with $\nu_{\bullet} = \pm (0,0,1) = \pm e_3$ we see that
\begin{equation*}
\mathcal{N}_{\bullet}u_{\bullet} = B_{\bullet}(x,D_x)u_{\bullet},
\end{equation*}
where the matrix operator $B_{\bullet} \in \text{Diff}^1(\Omega)$ is defined as
\begin{multline*}
B_{\bullet}(x,D_x) u_{\bullet}
= \bmat 0 &0 &\mu^{(\pm)}\p_{x_1}\\0&0&\mu^{(\pm)}\p_{x_2}\\\lambda^{(\pm)}\p_{x_1}&\lambda^{(\pm)}\p_{x_2}&0 \emat u_{\bullet}\\
\pm \bmat \mu^{(\pm)}&0&0\\ 0&\mu^{(\pm)} &0\\0&0&(\lambda^{(\pm)}+2\mu^{(\pm)})\emat \p_{x_3} u_{\bullet},
\quad \mbox{on }\Gamma,
\end{multline*}
where the sign $(\pm)$ in the above expression changes according to the sign of $\nu_{\bullet}$.
%
Since we are working only at the boundary, we will sometimes use the operators $B_{\bullet}$ and $\rho_\Gamma \circ B_{\bullet}$ interchangeably where $\rho_\Gamma$ is restriction to the interface.

First, we work with the case $(u_{\bullet})_0$ i.e. the term of $u_{\bullet}$ which are homogeneous of order $0$ in $\abs{\xi}$.
We also compute
\begin{equation*}
B_{\bullet}(x,D_x)e^{i\phi}a = e^{i\phi}(B_{\bullet}(x, \p_x \phi)a + B_{\bullet}(x,D_x)a).
\end{equation*}
It is convenient to introduce the shorthand $B(\phi_{\bullet,\PS})$ as the bundle endomorphism $B_{\bullet}(x, \p_x \phi_{\bullet,\PS}).$
Note that the transmission conditions of different order of homogeneity should be dealt separately.

Here we discuss the case of the $0$-th order transmission conditions on the interface $\Gamma$. The higher order transmission conditions have been discussed in the later subsections.
For $J=0$, using the form of the parametrix \[(u)_0 = \left(e^{\phi_{\bullet,P}}(a_P)_0 + e^{\phi_{\bullet,S}}(a_S)_0\right),\]
we form the $3 \times 3$ matrix $S_\bullet = [N_\bullet \abs{ N_{1,\bullet} }N_{2,\bullet}]$, where $N_\bullet$, $N_{1,\bullet}$, $N_{2,\bullet}$ are as in \eqref{eq_1} and \[(\mathcal{A}_\bullet)_0= \col{(\alpha_{\bullet})_0\\(\alpha_{1,\bullet})_0\\ (\alpha_{2,\bullet})_0}.\]
It is convenient to define
\begin{equation*}
\trac_{\bullet} = \bmat  l & l & l \\ B(\phi_{\bullet,P})N_\bullet & B(\phi_{\bullet,S})N_{1,\bullet} & B(\phi_{\bullet,S})N_{2,\bullet}\\l&l&l \emat,
\end{equation*}
where this $3 \times 3$ matrix is $0$-th order in the parameters.
Since, $\trac_{\bullet}$ is of order $1$ in $\abs{\xi}$, therefore,
the transmission conditions in \eqref{elastic_transmission} become
\begin{align*}\label{E_T_2}
S_I(\mathcal{A}_{I})_0 + S_R (\mathcal{A}_{R})_0 =& S_T (\mathcal{A}_{T})_0\\
\trac_{I} (\mathcal{A}_I)_0 + \trac_{R} (\mathcal{A}_R)_0
=& \trac_{T} (\mathcal{A}_T)_0
+ \left(B_{T}(x,D_x)\left(a_{T,P}+a_{T,S}\right)\right)_{1}\\
&- \left(B_{I}(x,D_x)(a_{I,P} + a_{I,S})\right)_{1}\\
&- \left(B_{R}(x,D_x)(a_{R,P} + a_{R,S})\right)_{1}.
\end{align*}

Since $B_{\bullet}(x,D_x)\left(a_{\bullet,P/S}\right)_J$ is homogeneous of order $J$ in $\abs{(\tau,\xi)}$, hence, $\left(B_{\bullet}(x,D_x)a_{\bullet,P/S}\right)_1 = 0$.
Therefore, the elastic transmission conditions for $J=0$ implies
\begin{equation}\label{e: elastic 0'th order trans conditions}
\bmat -S_R & S_T \\- \trac_R & \trac_T \emat \col{ (\mathcal{A}_R)_0 \\
(\mathcal{A}_T)_0} = \col{ S_I (\mathcal{A}_I)_0 \\ \trac_I (\mathcal{A}_I)_0}
\text{ on }\Gamma.
\end{equation}

\subsection{Parameters at the interface}
We start from the transmission conditions \eqref{e: elastic 0'th order trans conditions}. Observe that this is not quite the same situation as in the acoustic case \eqref{zeroth_order_wave}, since the given reflection coefficient is not $\alpha_R$ but rather $S_R \alpha_R$.
However, note that $S_R$ is completely determined by the material parameters above the interface, i.e. $\Gamma_-$ that we have access to. Hence, we may assume instead of $R = S_R \alpha_R$, we are indeed given $\alpha_R$, and the goal is to determine the material parameters below the interface. However, this is a calculation already done in \cite{CHKUElastic}. To make the connection, we will write using the ansantz, $(\mathcal{A}_R)_0 = R(\mathcal A_I)_0, (\mathcal A_T)_0 = T(\mathcal A_I)_0$ with $R, T$ being $3\times 3$ matrices of symbols. Then
\begin{equation*}
\col{(\mathcal{A}_R)_0 \\ (\mathcal{A}_T)_0} = \bmat -S_R & S_T \\- \trac_R & \trac_T \emat \col{ S_I \\ \trac_I}(\mathcal{A}_I)_0.
\end{equation*}
The symbols $R, T$ are exactly those computed in \cite{CHKUElastic}. Since $(\mathcal{A}_I)_0$ can be anything, we indeed recover $R$.

\begin{lemma}\label{lem_1}
	Let $\rho^{(-)}$, $c_S^{(-)}$, $c_P^{(-)}$ be known on $\Omega_{-}$. Then the knowledge of $R$ at the interface $\Gamma$ uniquely determines $\rho^{(+)}$, $c_P^{(+)}$ and $c_S^{(+)}$ on $\Gamma$.
\end{lemma}
\begin{proof}
	From \cite[Appendix A]{CHKUElastic}, the $r_{33}$ entry of $R$ is given as
	\begin{equation}\label{key_2}
	r_{33} = \frac{\mu^{(-)}\xi_{3,I,S} - \mu^{(+)}\xi_{3,T,S}}{ \mu^{(-)}\xi_{3,I,S} + \mu^{(+)}\xi_{3,T,S}}.
	\end{equation}
For the term $r_{33}$, we are in the same situation as in the Lemma \ref{l: recover 2 params from 0'th order reflect} for the acoustic case and using the same calculations done in the proof we recover $\rho^{(+)},c_S^{(+)}$ using just two values of $(\abs{\xi'}/\tau)$.
For the completion of the article we present the proof here.
We denote $f = \xi_{3,T,S}/\xi_{3,I,S}$, $a = \mu^{(-)}$, $c = \mu^{(+)}$.  Note that $f = f(\abs{\xi'}/\tau)$ i.e. it is a function of the parameter $\abs{\xi'}/\tau$ while $a,c$ only depend on $x$.
Now, assume that $r_{33} = \widetilde{r}_{33}$ and $\mu^{(-)} = \tilde{\mu}^{(-)}$ (i.e. $a=\tilde{a}$) on $\Gamma$ we obtain
\begin{align}\label{key_3_1}
\frac{af-c}{af+c} = \frac{a\tilde f-\tilde c}{a\tilde f+\tilde c},\quad
\Leftrightarrow \quad a\tilde c f = ac \tilde f, \quad
\Leftrightarrow \quad \frac{c}{\tilde c} = \frac{\tilde f}{f}.
\end{align}
Varying $\abs{\xi'}/\tau$ and keeping everything else constant, we get
\begin{equation*}
\frac{c}{\tilde c} = \frac{\tilde f_1}{f_1},
\end{equation*}
where $f_1$ is $f$ evaluated at different value of $\abs{\xi'}/\tau$. Thus,
\begin{align*}
\frac{\tilde f}{f}=\frac{\tilde f_1}{f_1}\quad
\Leftrightarrow \quad
\frac{(\tilde c^{(+)}_S)^{-2} - b^2}{(c^{(+)}_S)^{-2} - b^2}
= \frac{(\tilde c^{(+)}_S)^{-2} - b_1^2}{(c^{(+)}_S)^{-2} - b_1^2},
\end{align*}
where we write $b = \abs{\xi'}/\tau$, $b_1 = \abs{\xi'_1}/\tau_1$ and observe that $c^{(-)}_S = \tilde c^{(-)}_S$ on $\Gamma$.

Cross multiplying we get the algebraic equation
\begin{equation*}
((\tilde c^{(+)}_S)^{-2}- ( c^{(+)}_S)^{-2})(b^2 - b_1^2) = 0.
\end{equation*}
Note that, as long as we pick $b_1 \neq \pm b$, we recover $c_S^{(+)} = \tilde{c}_S^{(+)}$.
Then going back to \eqref{key_3_1} one gets $c=\tilde{c}$, that is $\mu^{(+)} = \tilde \mu^{(+)}$ on $\Gamma$.
Finally, from $c_S^{(+)} = \tilde{c}_S^{(+)}$ and $\mu^{(+)} = \tilde \mu^{(+)}$ we obtain $\rho^{(+)} = \tilde{\rho}^{(+)}$.

The other entries of $R$ are can be used to recover the remaining parameter $c_P^{(+)}$ with the analogous argument.
\end{proof}

So far we have seen that from the knowledge of the $0$-th order parameters on $\Gamma$ and the parameters in $\Omega_{-}$ we can uniquely determine $c_P^{(+)}$, $c_S^{(+)}$ and $\rho^{(+)}$ on $\Gamma_{+}$.

\subsection{Recovery of the derivatives of the material parameters at the interface}
In this subsection we determine the $J$-th order derivatives of the material parameters at the interface $\Gamma$ from the $J$-th transmission conditions.
We first establish a relation between the $J$-th reflection asymptotic term on $\Gamma_{-}$ with the Neumann data of the $J$-th asymptotic of the transmitted waves on $\Gamma_{+}$. Then we study the relation between the lower order asymptotic terms of the transmitted rays and the higher order derivatives of the material parameters at $\Gamma$.
We observe that the calculations for the higher order derivatives of the parameters at the interface $\Gamma$ is similar to the calculations done in \cite[Section 3]{RachBoundary}. We try to use similar notations, wherever possible, to draw a relation between the two articles.

{\bf Notation:} We denote by $R_j$ the terms depending on
\begin{enumerate}
	\item[$\bullet$]
	normal derivatives of $c_P$, $c_S$, $\rho$ of order at most $j$, and
	\item[$\bullet$]
	quantities determined by the transmission conditions \eqref{e: elastic 0'th order trans conditions}, \eqref{e: elastic J'th order trans conditions} in $\Gamma$ for $J=0,-1, \dots,$ $1-j$.
\end{enumerate}

\begin{lemma}\label{lemma_4}
	If $(u_R)_j = (\tilde{u}_R)_{j}$, $(u_I)_j = (\tilde{u}_I)_{j}$, $c_{\PS} = \widetilde{c}_{\PS}$ and $\rho = \widetilde{\rho}$  on $\Omega_{-}$, for $j = -1,-2,\dots,$ $J$, then
	\begin{equation*}
	\left(\p_{x_3}u_{T}\right)_{J+1}
	= \left(\p_{x_3}\widetilde{u}_{T}\right)_{J+1},\quad\mbox{for }J\leq 0 	\quad \mbox{on }\Gamma_{+}.
	\end{equation*}
\end{lemma}
\begin{proof}
 We recall the elastic transmission conditions (\ref{elastic_transmission}) and the boundary Neumann data as
\begin{align}\label{Neumann_data_J-th}
\mathcal{N}u_{\bullet} =& B(x,\phi_{\bullet,P})a_{\bullet,P} + B(x,\phi_{\bullet,S})a_{\bullet,S} + B(x,\nabla_x)\left(a_{\bullet,P}+a_{\bullet,S}\right) \\
=& \sum_{\star=P/S}B(x,\phi_{\bullet,\star})a_{\bullet,\star} + \bmat 0 &0 &\mu\p_{x_1}\\0 &0 &\mu\p_{x_2}\\\lambda\p_{x_1} &\lambda\p_{x_2} &0\\ \emat \sum_{\star=P/S} a_{\bullet,\star}
\\\nonumber
&\qquad \qquad \qquad \qquad \qquad + \bmat \mu &0 &0\\ 0&\mu&0\\ 0&0&(\lambda+2\mu)\emat \sum_{\star=P/S} \p_{x_3} a_{\bullet,\star}.
\end{align}
Observe that, $\mathcal{N}^{(-)}u_{I/R}$ is completely determined in $\Gamma_{-}$ from the knowledge of $\phi_{I/R,\PS}$, $c_{\PS}$, $\rho$ and $a_{I/R,\PS}$ on $\Omega_{-}$ except the term $\p_{x_3}a_{I/R,\PS}$. That is, one may write $\mathcal{N}^{(-)}u_{I/R} = R_0 + P^{(-)}\p_{x_3} u_{I/R}$, where $P^{(\pm)}$ is the diagonal matrix diag$(\mu^{(\pm)},\mu^{(\pm)},(\lambda^{(\pm)}+2\mu^{(\pm)}))$.

Note that $\p_{x_k}\phi_{\bullet,\PS} = \xi_k$, for $k=1,2$ and $\p_{x_3}\phi_{\bullet,\PS} = \xi_{3,\bullet,\PS}$ on $\Gamma$.
Therefore, the $J$-th order transmission conditions become
\beq
\begin{aligned}\label{e: elastic J'th order trans conditions}
(u_T)_{J} =& (u_I)_{J} + (u_R)_{J}\\
\left(\mathcal{N}_{T}(x,\xi)u_{T}\right)_{J+1}
=&\left(\mathcal{N}_{I}(x,\xi)u_{I}\right)_{J+1} + \left(\mathcal{N}_{R}(x,\xi)u_{R}\right)_{J+1}
\end{aligned}
\eeq
Now if $(u_I)_J = (\widetilde{u}_I)_J$ and $(u_R)_J = (\widetilde{u}_R)_J$ on $\Gamma$ for $J \leq -1 $, then
\begin{equation}
(u_T)_J = (u_I)_J + (u_R)_J = (\widetilde{u}_I)_J + (\widetilde{u}_R)_J = (\widetilde{u}_T)_J, \qquad \mbox{on }\Gamma.
\end{equation}
Moreover, $\Pi_{\PS}(u_{\bullet})_J = \Pi_{\PS}(\tilde{u}_{\bullet})_J$ on $\Gamma$ since $\Pi_{\PS}$ is a $R_0$ quantity (see \eqref{Symbol_mode-projector}).
Therefore, from \eqref{Neumann_data_J-th} along with the fact that $(u_{I/R})_{J+1}=(u_{I/R})_{J+1}$ on $\Gamma$ we obtain
\begin{equation*}
\left(\mathcal{N}_Iu_{I}\right)_{J+1} = \left(\widetilde{\mathcal{N}}_I \widetilde{u}_{I}\right)_{J+1}
\qquad \mbox{and}\quad
\left(\mathcal{N}_Ru_{R}\right)_{J+1} = \left(\widetilde{\mathcal{N}}_R \widetilde{u}_{R}\right)_{J+1}
\qquad \mbox{on }\Gamma.
\end{equation*}
Now, let $\mathcal{N}$ be the Neumann derivative for the parameters $\widetilde{\lambda}, \widetilde{\mu}, \widetilde{\rho}$ on $\Gamma$.
We readily obtain
\begin{multline}
\left(\mathcal{N}_{T} u_{T}\right)_{J+1}
=\left(\mathcal{N}_{I} u_{I}\right)_{J+1}
+ \left(\mathcal{N}_{R} u_{R}\right)_{J+1}
= \left(\widetilde{\mathcal{N}}_{I} \widetilde{u}_{I}\right)_{J+1} + \left(\widetilde{\mathcal{N}}_{R} \widetilde{u}_{R}\right)_{J+1}
\\= \left(\widetilde{\mathcal{N}}_{T} \widetilde{u}_{T}\right)_{J+1},
\qquad \mbox{on }\Gamma.
\end{multline}
Note that, from $\Pi_{\PS}(u_T)_{J+1}=\Pi_{\PS}(\widetilde{u}_T)_{J+1}$ on $\Gamma$ and \eqref{Neumann_data_J-th} we see
\begin{multline*}
0 = \left(\mathcal{N}_{T} u_T\right)_{J+1} - \left(\widetilde{\mathcal{N}}_{T} \widetilde{u}_T\right)_{J+1}
= P^{(+)}\left(\p_{x_3}u_{T}\right)_{J} - \widetilde{P}^{(+)}\left(\p_{x_3}\widetilde{u}_{T}\right)_{J}
\\= P^{(+)}\left[\left(\p_{x_3}u_{T}\right)_{J} - \left(\p_{x_3}\widetilde{u}_{T}\right)_{J}\right],
\end{multline*}
on $\Gamma$. The last identity holds due to the fact that Lemma \ref{lem_1} asserts $P^{(+)} = \widetilde{P}^{(+)}$ on $\Gamma$.
Therefore, we essentially obtain $\left(\p_{x_3}\widetilde{u}_{T}\right)_{J} = \left(\p_{x_3}u_{T}\right)_{J}$ on $\Gamma$ since $P^{(+)}(x)$ is invertible.
\end{proof}

\begin{rem}
A similar lemma can be proved if the transmitted wave fields $(u_T)_J$ are known instead of $(u_R)_J$ on $\Gamma$. That is, if we know the elastic parameters on $\Omega_{-}$, $(u_I)_J$ on $\Gamma_{-}$ and $(u_T)_J$ on $\Gamma_{+}$, then one can determine the reflected wave fields $(\p_{x_3}u_R)_J$ on $\Gamma_{-}$.
\end{rem}

In the rest of the section we will show that knowing $\left(\p_{x_3}u_T\right)_J$ on $\Gamma_{+}$ implies knowing $\p_{x_3}^{\abs{J}+1}c_{P/S}^{(+)}$ and $\p_{x_3}^{\abs{J}+1}\rho^{(+)}$ on $\Gamma_{+}$.
We start with the following computation whose proof follows from \cite[Proposition 3.1]{RachBoundary}. One can describe $\left(\p_{x_3}u_{\bullet}\right)_J\restriction_{\Gamma_{+}}$ as

\begin{multline}\label{eq3.1}
\left(\p_{x_3}u_{\bullet}\right)_J
= i(\gamma_{2,\bullet})_{J-1} \left[ \xi_{3,\bullet,P}\left(M_2 + \frac{\abs{\xi'}^2(\xi_{3, \bullet,S} 
- \xi_{3,\bullet,P})}{\abs{\xi'}^2 + \xi_{3, \bullet,P}\xi_{3, \bullet,S}} \xi_{\bullet,P}\right)\right.
\\\left.
- \xi_{3,\bullet,S}\left(\frac{\abs{\xi'}\,\abs{\xi_{\bullet,P}}^2 \,\abs{\xi_{\bullet,S}}}{\abs{\xi'}^2 + \xi_{3, \bullet,P}\xi_{3, \bullet,S}} N_2\right)\right]\\
\qquad - (\gamma_{\bullet})_{J-1} \left[ \xi_{3,\bullet,P}\left(\frac{\abs{\xi_{\bullet,S}}^2}{\abs{\xi'}^2 + \xi_{3, \bullet,P}\xi_{3, \bullet,S}} \xi_{\bullet,P}\right)
\right.\\
\left.-\xi_{3,\bullet,S}\left(\xi_{\bullet,S} + \frac{\abs{\xi'}\,\abs{\xi_{\bullet,S}}(\xi_{3,\bullet,S} - \xi_{3,\bullet,P})}{\abs{\xi'}^2 + \xi_{3, \bullet,P}\xi_{3, \bullet,S}} N_2\right)\right]\\
\vspace{2 mm} \qquad + i(\gamma_{1,\bullet})_{J-1}\left[ \left(\xi_{3,\bullet,P}-\xi_{3,\bullet,S}\right) M_1 \right]
+ \left[\p_{x_3}\left(\gamma_{\bullet}\right)_J\right]M
+ \left[\p_{x_3}\left(\gamma_{1,\bullet}\right)_J\right]M_1\\
\qquad +\left[\p_{x_3}\left(\gamma_{2,\bullet}\right)_J\right]M_2
 + \left[\p_{x_3}\left(\alpha_{\bullet}\right)_J\right]N
+ \left[\p_{x_3}\left(\alpha_{1,\bullet}\right)_J\right]N_1 \\
\qquad + \left[\p_{x_3}\left(\alpha_{2,\bullet}\right)_J\right]N_2
+ \left(\gamma_{\bullet}\right)_J\left[\p_{x_3}M\right]
+ \left(\gamma_{1,\bullet}\right)_J\left[\p_{x_3}M_1\right]\\
\qquad + \left(\gamma_{2,\bullet}\right)_J\left[\p_{x_3}M_2\right]
+ \left(\alpha_{\bullet}\right)_J\left[\p_{x_3}N\right]
+ \left(\alpha_{1,\bullet}\right)_J\left[\p_{x_3}N_1\right]
+ \left(\alpha_{2,\bullet}\right)_J\left[\p_{x_3}N_2\right].
\end{multline}

Before going into the full general case for recovering $|J|$-th order derivatives of the elastic parameters on the interface, we consider the case for $|J|=1$ in the following proposition.
\begin{prop}\label{prop 5}
The terms $\p_{x_3}c_P^{(+)}$, $\p_{x_3}c_S^{(+)}$ and $\p_{x_3}\rho^{(+)}$ are uniquely determined on $\Gamma_{+}$ from the knowledge of $(u_R)_0$ and $(u_R)_{-1}$ on $\Gamma_{-}$.
\end{prop}
\begin{proof}
We start with the following relation, obtained from a similar calculation done in \cite[Equation (64)]{RachBoundary}, given as
\begin{equation}\label{key_e_5}
	\frac{M_1}{\abs{M_1}^2}\cdot \p_{x_3}(u_{T})_0
	= -\left(\p_{x_3}\log\sqrt{\rho^{(+)}} + \frac{1}{(2c_s^{(+)})^2 f_S}\p_{x_3}\log c_S^{(+)}\right)(\alpha_{1,T})_0 + R_0,
\end{equation}
where \[f_S = f(\abs{\xi'}/\tau) = \left(\frac{1}{(c_S^{(+)})^2} - \frac{\abs{\xi'}}{\tau}\right)\] is a $R_0$ quantity.
Since, $M_1$ is a $R_0$ quantity, hence, from Lemma \ref{lemma_4} for $J=0$ and \eqref{key_e_5} we get
\begin{equation}\label{key_e_6}
	\p_{x_3}\log\sqrt{\frac{\rho^{(+)}}{\tilde{\rho}^{(+)}}}
	= \frac{1}{2(c_s^{(+)})^2 f_S}\left(\p_{x_3}\log \tilde{c}_S^{(+)} - \p_{x_3}\log c_S^{(+)}\right), \qquad\mbox{on }\Gamma_{+}.
\end{equation}
Similar to the case of acoustic waves, we fix $x$ and consider two different values of $(\abs{\xi'}/\tau)$ to have two different quantities $f_S$ and $f_S^{(1)}$.
Thus, we obtain
\begin{equation*}
	\left(f_S - f^{(1)}_S \right)\left(\p_{x_3}\log \tilde{c}_S^{(+)} - \p_{x_3}\log c_S^{(+)}\right) = 0
	\quad\mbox{on }\Gamma.
\end{equation*}
Hence, $\p_{x_3}\log \tilde{c}_S^{(+)} = \p_{x_3}\log c_S^{(+)}$ on $\Gamma_{+}$. From \eqref{key_e_6} we get $\p_{x_3}\log \sqrt{\tilde{\rho}^{(+)}}\restriction_{\Gamma_{+}}$ $= \p_{x_3}\log \sqrt{\rho^{(+)}}\restriction_{\Gamma_{+}}$.
	
Now, to recover $\p_{x_3}c_P^{(+)}\restriction_{\Gamma_+}$, we observe that a similar calculation as in \cite[Proposition 3.8]{RachBoundary} gives us
\beq
\begin{aligned}\label{key_e_8}
	\frac{M_2}{\abs{M_2}}\cdot \p_{x_3}(&u_{T})_0
	\\=&  \left((\p_{x_3}\log c_P^{(+)})\left[ i\frac{(\xi_{3,T,P} - \xi_{3,T,S})} {2\xi_{3,T,P}^2} + \frac{\abs{\xi'}}{\xi_{3,T,P}}\right]\right.
	\\
&-(\p_{x_3}\log c_S^{(+)}) \left[i\frac{4 (c_S^{(+)})^2(\xi_{3,T,P} - \xi_{3,T,S})\abs{\xi'}}{\abs{\xi_{T,P}}^2}\right] \\
	&\left. - \left(\p_{x_3}\log \sqrt{\rho^{(+)}}\right) \left[ i\left(1- \frac{2(c_S^{(+)})^2}{(c_{\lambda+\mu}^{(+)})^2}\right)\frac{(\xi_{3,T,P} - \xi_{3,T,S})\abs{\xi'}}{\abs{\xi_{T,P}}^2}\right] \right)(\alpha_{1,T})_0 \\
	&\qquad  + R_0,
	\end{aligned}
\eeq
where $c_{\lambda+\mu}^{(\pm)}:= \sqrt{(c_P^{(\pm)})^2 - (c_S^{(\pm)})^2} = \sqrt{\frac{\lambda^{(+)} + \mu^{(+)}}{\rho^{(+)}}}$.
Since \[\p_{x_3}\log \tilde{c}_S^{(+)}\restriction_{\Gamma_{+}} = \p_{x_3}\log c_S^{(+)}\restriction_{\Gamma_{+}} ,\qquad  \p_{x_3}\log\sqrt{\tilde{\rho}^{(+)}}\restriction_{\Gamma_{+}}=\p_{x_3}\log\sqrt{\rho^{(+)}}\restriction_{\Gamma_{+}},\] then $(a_R)_{-1} = (\tilde{a}_R)_{-1}$ on $\Gamma_{-}$ implies
\begin{multline*}
	(\p_{x_3}\log c_P^{(+)})\left[ i\frac{(\xi_{3,T,P} - \xi_{3,T,S})} {2\xi_{3,T,P}^2} + \frac{\abs{\xi'}}{\xi_{3,T,P}}\right]\\
	= (\p_{x_3}\log \tilde{c}_P^{(+)})\left[ i\frac{(\xi_{3,T,P} - \xi_{3,T,S})} {2\xi_{3,T,P}^2} + \frac{\abs{\xi'}}{\xi_{3,T,P}}\right] 
	\quad \mbox{on }\Gamma,\\
	\text{which implies } \quad
	\p_{x_3} c_P^{(+)}\restriction_{\Gamma_{+}} =\ \p_{x_3} \tilde{c}_P^{(+)}\restriction_{\Gamma_{+}}  \quad\quad
	\left[\because \frac{(\xi_{3,T,P} - \xi_{3,T,S})} {2\xi_{3,T,P}^2} \neq i\frac{\abs{\xi'}}{\xi_{3,T,P}}\right].
\end{multline*}
\end{proof}

Next we define the quantities
\begin{align*}
A_P =& \bmat 1 &-(c_P/\tau)^{3}\\ \\ 0 &-2c_p^2\xi_{3,\bullet,P} \emat,
\quad C_P = \bmat(c_P/\tau)^{2}\left(\frac{c_S^2}{c_{\lambda+\mu}^2}+\frac{c_P^2 \abs{\xi'}^2}{\tau^2}\right) &0\\ \\ -c_{\lambda+\mu}^2\frac{c_P}{\tau}\abs{\xi'}^2\xi_{3,\bullet,P} &0 \emat,\\
B_P =& \bmat \frac{2ic_S^2c_P^2\xi_{3,\bullet,P}}{c_{\lambda+\mu}^2\tau^2} &\frac{c_P^5\xi_{3,\bullet,P}}{\tau^5}\\ \\ -c_{\lambda+\mu}^2\abs{\xi'}^2\frac{\tau}{c_p} &c_{\lambda+\mu}^2\frac{c_P\xi_{3,\bullet,P}}{\tau}+c_S^2 \emat,\quad
\\
&
D_P^{J+1} = \col{-(\gamma_{2,\bullet})_{J+1}\frac{c_P^2 c_S^2}{c_{\lambda+\mu}^2\tau^2} - \frac{(\alpha_{\bullet})_{J+1}c_{P}^3 \left(\frac{c_{\lambda+\mu}^2c_P^2\abs{\xi'}^2}{\tau^2} + c_S^2\right)}{c_{\lambda+\mu}^2\tau^3\xi_{3,\bullet,P}}\\ \\
	(\gamma_{2,\bullet})_{J+1}c_S^2 \frac{\tau \abs{\xi'}^2}{c_P\xi_{3,\bullet,P}}- (\alpha_{\bullet})_{J+1}c_{\lambda+\mu}^2\frac{c_P^2\abs{\xi'}^2}{\tau^2}}.
\end{align*}
Observe that $A_P$, $B_P$, $C_P$ are $R_0$ terms and $D_{P}^{J+1}$ are $R_{\abs{J}-1}$ terms, for $J = 0,-1,-2,\dots$.
We have the following recursive relation from \cite[Theorem 3.7]{RachBoundary} as
\begin{prop}\label{Prop_3.7}
For P waves, we have the following recurrence relation for $(\gamma_{2,\bullet})_{J-1}$ and \newline $\p_{x_3}(\alpha_{\bullet})_{J}$ as
\begin{multline*}
A_P\col{(\gamma_{2,\bullet})_{J-1}\\\p_{x_3}(\alpha_{\bullet})_{J}}
= B_P\p_{x_3}\col{(\gamma_{2,\bullet})_{J}\\\p_{x_3}(\alpha_{\bullet})_{J+1}} + C_P\p_{x_3}^2\col{(\gamma_{2,\bullet})_{J+1}\\\p_{x_3}(\alpha_{\bullet})_{J+2}}
\\
+ D_P^{J+1}\left(\p_{x_3}^2\log c_{\bullet,P}\right) + R_{\abs{J}+1},
\end{multline*}
for  $J\leq-1$.
\end{prop}

Using the recurrence relation in Proposition \ref{Prop_3.7} we state the following lemma.
\begin{lemma}\label{Lem_3.12}
	$(\gamma_{2,\bullet})_{J-1}$ and $\p_{x_3}(\alpha_{\bullet})_J$ can be written in terms of $\p_{x_3}^{1+\abs{J}} \log c_P, \p_{x_3}^{1+\abs{J}} \log c_s$ and
	$\p_\nu^{1+\abs{J}} \log \sqrt \rho$.
	In fact,
	\begin{multline*}
	\col{ (\gamma_{2,\bullet})_{J-1}\\ \p_{x_3}(\alpha_{\bullet})_J}
	= (I \ 0) \cdot \mathcal M_J \cdot \mathcal M
	\cdot \p_{x_3}^{1+\abs{J}} \col{ \log c_P \\ \log c_S \\ \log \sqrt \rho}
	(\alpha_{\bullet})_0 + R_{\abs{J}}, \\
	\qquad	\mbox{for } J=-1, -2,\dots,
	\end{multline*}
	where
	\begin{align*}
	\mathcal{M}
	=& \col{ A_p^{-1}B_p \\ I}\mathcal M_{\gamma_2,\alpha} + \col{I\\0}
	\bmat \bmat A_p^{-1}D_p^0 \emat & \bmat 0&0\\0&0 \emat \emat,
	\qquad I = \bmat 1 &0\\ 0 &1 \emat,\\
	\mathcal{M}_J
	=& \col{ A_p^{-1}B_p & A_p^{-1}C_p\\ I &0}^{\abs{J}-1},
\end{align*}
in which,
\begin{align*}
	&\col{ (\gamma_2)_{-1}\\ \p_\nu(\alpha)_0}
	=\mathcal{M}_{\gamma_2,\alpha}\cdot \p_{x_3}\col{ \log c_P \\ \log c_S \\ \log \sqrt \rho}(\alpha_{\bullet})_0,\\
	\mathcal{M}_{\gamma_2,\alpha} =& \bmat -\frac{c_{P,\bullet}^2}{2\tau^2\xi_{3,\bullet,P}^2} &\frac{4ic_{P,\bullet}^3c_{S,\bullet}^2}{\tau^3c_{\lambda+\mu,\bullet}^2} &i\left(1-\frac{2c_{S,\bullet}^2}{c_{\lambda+\mu,\bullet}^2}\right)\frac{c_{p,\bullet}^3}{\tau^3}\\
		-\frac{1}{2}\left(1- \frac{\abs{\xi'}^2}{\xi_{3,\bullet,P}^2}\right) &0 &-1 \emat.
\end{align*}
	
\end{lemma}
\noindent The proof of the above Lemma follows from similar calculations done in \cite[Lemma 3.12]{RachBoundary}.

%

\begin{lemma}\label{Lem_3.13_1}
One can determine $\p_{x_3}^{\abs{J}}c^{(\pm)}_{S}$ and $\p_{x_3}^{\abs{J}}\rho^{(\pm)}$ on $\Gamma_{-}$ from the knowledge of $(u_R)_{j}$, $c_{\PS}$ and $\rho$ on $\Omega_{-}$, for $j = 0,-1,\dots,J-1$.
\end{lemma}
\begin{proof}
From the equation \eqref{eq3.1} and Lemma \ref{Lem_3.12} we obtain the following relation
\begin{multline}\label{equation_67_1}
\left(\p_{x_3}u_{T}\right)_J \cdot \frac{M_1}{\abs{M_1}^2}
\\= -\left(\frac{i}{2\xi_{3,T,S}}\right)^{\abs{J}} \col{0\\\frac{1}{2}\left(1-\frac{\abs{\xi'}}{\xi_{3,T,S}^2}\right)\\1} \cdot \p_{x_3}^{\abs{J}+1}\col{ \log c_P \\ \log c_S \\ \log \sqrt \rho}(\alpha_{1,T})_0 + R_{\abs{J}}.
\end{multline}
From Lemma \ref{lemma_4} we get $\left(\p_{x_3}u_{T}\right)_J = \left(\p_{x_3}\widetilde{u}_{T}\right)_J$ and fact that $M_1$ is a $R_0$ quantity we obtain
\beq
\begin{aligned}\label{J-th_relation_1}
\left(\frac{i}{2\xi_{3,T,S}} \right)^{\abs{J}}
&\left[
\frac{1}{2}(\p_{x_3}^{1+\abs{J}} \log c_S^{(+)})\left(1 - \frac{\abs{\xi'}^2}{(\xi_{3,T,S})^2}  \right) + (\p_{x_3}^{1+\abs{J}}\log \sqrt{\rho^{(+)}})
\right]\\
&=
\left(\frac{i}{2\xi_{3,T,S}} \right)^{\abs{J}}
\left[
\frac{1}{2}(\p_{x_3}^{1+\abs{J}} \log \tilde{c}_S^{(+)})\left(1 - \frac{\abs{\xi'}^2}{(\xi_{3,T,S})^2}  \right)\right. \\
&\left. \qquad \qquad \qquad \qquad \qquad \qquad + (\p_{x_3}^{1+\abs{J}}\log \sqrt{\tilde{\rho}^{(+)}})\right],
\quad\mbox{on }\Gamma_{+}.
\end{aligned}
\eeq
Varying $(1+\abs{\xi'}^2/\xi_{3,T,S}^2) = (c_S^{(+)})^{-2}f_S^{-2}$ as in the proof of Proposition \ref{prop 5} we get
\begin{multline*}
\frac{1}{(c_S^{(+)})^2f_S^2}(\p_{x_3}^{1+\abs{J}} \log c_S^{(+)})
- \frac{1}{(\tilde{c}_S^{(+)})^2f_S^2}(\p_{x_3}^{1+\abs{J}} \log \tilde{c}_S^{(+)})\\
= \frac{1}{(c_S^{(+)})^2(f_S^{(1)})^2}(\p_{x_3}^{1+\abs{J}} \log c_S^{(+)})
- \frac{1}{(\tilde{c}_S^{(+)})^2(f_S^{(1)})^2}(\p_{x_3}^{1+\abs{J}} \log \tilde{c}_S^{(+)}),\\
\text{so that}\quad
(c_S^{(+)})^2\left(f_S^2 - (f_S^{(1)})^2\right)\left( \p_{x_3}^{1+\abs{J}} \log c_S^{(+)} - \p_{x_3}^{1+\abs{J}} \log \tilde{c}_S^{(+)} \right) = 0
\quad\mbox{on }\Gamma_{+}.
\end{multline*}
Choosing $f_S \neq f_S^{(1)}$ we obtain $\p_{x_3}^{1+\abs{J}}c_S^{(+)}\restriction_{\Gamma_{+}} = \p_{x_3}^{1+\abs{J}}\tilde{c}_S^{(+)}\restriction_{\Gamma_{+}}$. Going back to \eqref{J-th_relation_1} we further obtain $\p_{x_3}^{1+\abs{J}}\rho^{(+)}\restriction_{\Gamma_{+}} = \p_{x_3}^{1+\abs{J}}\tilde{\rho}^{(+)}\restriction_{\Gamma_{+}}$.
\end{proof}

\begin{lemma}\label{l: higher order p speed from reflection}
One can determine $\p_{x_3}^{\abs{J}+1}c_P^{(+)}\restriction_{\Gamma_{+}}$ from the knowledge of $(u_R)_{j}\restriction_{\Gamma_{-}}$, for $j=0,-1,$ $\dots,J-1$, where $J\leq -1$.
\end{lemma}

\begin{proof}
In order to determine $\p_{x_3}^{\abs{J}+1}c_P^{(+)}$ on $\Gamma_{+}$ we go back to equation \eqref{eq3.1}, Lemma \ref{Lem_3.12} and observe that
\begin{multline*}
\p_{x_3}(u_{T})_J \cdot \frac{M_2}{\abs{M_2}^2}
=\left(\col{i(\xi_{3,T,P}-\xi_{3,T,S}) \\0 \\1 \\0} \mathcal{M}_J\mathcal{M}\right) \cdot \p_{x_3}^{\abs{J}+1}\col{\log c_P\\ \log c_S\\ \log \sqrt{\rho}}(\alpha_{T})_0 + R_{\abs{J}}, \\ \mbox{for } J=-1,-2,\dots.
\end{multline*}
Now we are in the exact same situation as in the proof of \cite[Theorem 3.13]{RachBoundary} and following the exact same calculations there one finally obtains $\p_{x_3}^{\abs{J}+1}\left(c_P^{(+)}\right) = \p_{x_3}^{\abs{J}+1}\left(\widetilde{c}_P^{(+)}\right)$ on ${\Gamma_{+}}$.
\end{proof}

\begin{proof}[Proof of Theorem \ref{t: Elastic case}]
The proof follows the same argument as the proof of Theorem \ref{Th_main_acoustic} in section \ref{s: lemmas and proof of acoustic thm} using the above lemmas in this section.
\end{proof}

\section{Extending the previous results to a non-flat interface}\label{s: nonflat case}
We will briefly show how the earlier proofs extend to the nonflat case. Essentially, the only changes are that lower order terms such as $(a_R)_J$, $J=-1,-2,\dots$ will contain terms involving the curvature of the interface, also known as the shape operator. However, when we try to determine $\p_\nu^j c_\PS$, any such curvature term in $(a_R)_{-j-1}$ will be a $R_j$ term and hence completely determined from the previous step in the induction argument. Hence, the proof will proceed with little change, and the formulas from the previous sections continue to hold. Nevertheless, it is worthwhile to do the calculation to see how the geometry of the interface is incorporated in the reflection operator. We do the main calculation in the acoustic case, but make it clear that similar calculation continue to hold in the elastic case.

\subsection*{Some geometric notation}
First, define boundary normal coordinates
\[
\tilde x = (\tilde x', \tilde x_3)
\]
near $\Gamma = \{ \tilde x_3 = 0\}$, here with respect to the Euclidean metric. Then $\Omega_-$ is given by $\tilde x_3 < 0$ and $\Omega_+$ is $\tilde x_3 > 0$ (see \cite{RachBoundary}). The directions of the boundary-normal-coordinate axes are given by the orthogonal vectors $\nabla_x \tilde x_1, \nabla_x \tilde x_2, \nabla_x \tilde x_3 = -\nu$. Here, we denote $\nu$ as the vector field that is normal to the interface when restricted to $\Gamma$. In semigeodesic coordinates, the Euclidean metric has the form $g = d\tilde x_3^2 + h(x, d\tilde x')$ where $h\restriction_\Gamma$ is the induced metric on $\Gamma$.

If $\tilde \xi'$ are the dual coordinates to $\tilde x'$, then $\xi_{tan} = \left( \frac{\p \tilde x'}{\p x}\right)^t \tilde \xi'$. We also have $\nabla_{tan} \phi = \left( \frac{\p \tilde x'}{\p x}\right)^t (\nabla_{\tilde x'} \phi)$, and $\nabla_x \phi = \nabla_{tan} \phi + \p_\nu \phi \nabla_x \tilde x_3$.
Similarly, $\xi_{3,\bullet}$ is defined as before using the $\tilde x$ coordinates: \[\p_\nu \phi_{\bullet}= \xi_{3,\bullet} =
\sqrt{  c_S^{-2}\abs{\p_t \phi_\bullet}^2-\abs{\nabla_{tan} \phi_\bullet }^2 } =
\sqrt{c_S^{-2}\abs{\tau}^2 - \abs{ \xi'_{tan}}^2}\] on $\Gamma$.
We also define a useful object for studying submanifolds.
\begin{definition}
Let $\Gamma$ be a surface, $p \in \Gamma$, and $\nu$ a smooth unit normal vector field defined along a neighborhood of $p$. The \emph{shape operator} is the map $S_\Gamma: T_p \Gamma \to T_p \Gamma$ defined by,
\[
S_\Gamma(X) = -\nabla_X \nu
\]
where $\nabla_X$ is the covariant derivative.
\end{definition}

\subsection*{Curvature contributions to the symbols}

We are ready to prove the following proposition. $a_R \sim \sum (a_R)_J$ will be defined as before, but we assume $\Gamma$ is a smooth interface and not necessarily flat.
\begin{prop}Assume $\Gamma$ is a smooth hypersurface.
Equation \eqref{e: (a_R)_J equation for acoustic} continues to hold for $(a_R)_J$ with $\p_{x_3}$ replaced by $\p_\nu$. That is,
\begin{multline*}
(a_R)_J
= -(-i/(2\xi_{3,T}))^J \left[(\p^{\abs{J}}_{\nu}\log \sqrt {\rho^{(+)}}) \right.
\\
\left. +\p^{\abs{J}}_{\nu}\log c^{(+)}_S\left(1 - \frac{(\p_t \phi_T)^2}{2c_S^2 \xi^2_{3,T}}\right)\right]\frac{(a_T)_{J+1}}{R_{\abs{J+1}}} + R_{\abs{J+1}}
\end{multline*}
Hence, Theorem \ref{Th_main_acoustic} continues to hold. Moreover, $R_{\abs{J+1}}$ above differs from $R_{\abs{J+1}}$ computed in \eqref{e: (a_R)_J equation for acoustic} by terms depending \emph{only} on $S_\Gamma$. Thus, the full reflection operator $R$ in the general case differs from $R$ in the flat case only by terms depending on $S_\Gamma$, i.e. the curvature of $\Gamma$. Similarly, in the elastic case Theorem \ref{t: Elastic case} and Corollary \ref{c: recover T from R} continue to hold as well.
\end{prop}
The second statement about curvature is nontrivial and requires a careful geometric argument since as seen in the above equation, $(a_R)_J$ involves higher order normal derivatives and in the non-flat case, it will involve higher order normal derivatives of quantities related to the curvature of $\Gamma$ (c.f. \eqref{e: (a)_J with higher curvature}). Nevertheless, we show that all such high order derivatives are still determined by just the curvature of $\Gamma$ (in fact, the eigenvalues of $S_\Gamma$) and no other information is needed to compute the full reflection operator.

\begin{proof}
 First, we consider the acoustic case since the elastic case will follow from analogous calculations. Our goal is to compute the full symbol of the reflection operator in the non-flat case and show that the only additional terms from that of the flat case done earlier are completely determined by the shape operator $S_\Gamma$.

Obsserve that
\begin{align*}
\nabla \cdot \mu \nabla_x \phi
&= \nabla_x \mu \cdot \nabla_x \phi+ \mu \nabla \cdot (\nabla_{tan} \phi + \p_\nu \phi \nabla_x \tilde x_3)\\
&= \nabla_x \mu \cdot \nabla_x \phi
+ \mu \div_x(\nabla_{tan}\phi) + \mu\nabla \p_\nu \phi \cdot \nabla_x \tilde x_3
+ \mu \p_\nu \phi \div_x(\nabla_x \tilde x_3) \\
&= \p_\nu \mu \p_\nu \phi + \mu R(\nabla_{tan}\phi)+
\mu \p^2_\nu \phi
+ \mu \p_\nu \phi H(x) + R_0
\end{align*}
where $H(x)$ is proportional to the mean curvature of the interface at $x$, which can be computed by taking the divergence of the normal vector field and is determined by the eigenvalues of $S_\Gamma$. $R(X) = \langle \nabla_\nu X, \nu \rangle$ will also be a term containing curvature. However, here, $R_0$ will be non-curvature terms with $0$ normal derivatives of the material parameters.

From \eqref{e: transport equations} and \eqref{e: Hamilton derivative to normal deriv} we obtain
\begin{multline*}
i c_S^2 \xi_{3,\bullet} \p_\nu (a)_0
= -c_S^2 \xi_{3,\bullet}\left[(\p_{\nu}\log \sqrt \rho)
- (\p_{\nu}\log c_S)\left(1 - \frac{(\p_t \phi_\bullet)^2}{2c_S^2 \xi^2_{3,\bullet}}\right)\right.
\\\left.+ H(x)/2 + \frac{R(\nabla_{tan}\phi)}{2\xi_{3,\bullet}}\right](a_\bullet)_0 + R_0.
\end{multline*}
so that
\begin{multline}
 \p_\nu (a_\bullet)_0 = -\left[(\p_{\nu}\log \sqrt \rho)
- (\p_{\nu}\log c_S)\left(1 - \frac{(\p_t \phi_\bullet)^2}{2c_S^2 \xi^2_{3,\bullet}}\right)\right.
\\\left.+ H(x)/2 + \frac{R(\nabla_{tan}\phi)}{2\xi_{3,\bullet}}\right](a_\bullet)_0 + R_0
\end{multline}
Note that the $R$ and $H$ terms are $R_0$ terms an have no normal derivatives of any material parameters so that Lemma \ref{l: acoustic partial_x_3 a_0 formula} continues to hold even in this non-flat case.

Using semigeodesic coordinates actually allows us to simplify the $R(\nabla_{tan}\phi)$ term.
Let $e_1, e_2, e_3$ denote the basis of vector fields corresponding to the coordinates $\tilde x_1, \tilde x_2, \tilde x_3$ with $e_3 = \nu$ being the normal vector.
Then
\begin{equation}
\langle \nabla_\nu\nabla_{tan}\phi_\bullet, \nu \rangle
=\langle \nabla_\nu(\sum_{j=1}^2\p_{\tilde x_j} \phi_\bullet e_j), \nu \rangle
=\sum_{j=1}^2\p_{\tilde x_j} \phi_\bullet\langle \nabla_{e_3} e_j, e_3 \rangle
=\sum_{j=1}^2\p_{\tilde x_j} \phi_\bullet \Gamma^3_{j3} =0
\end{equation}
where $\Gamma^3_{j3}$ are Christoffel symbols, and these ones vanish in semigeodesic coordinates \cite[Section 2.4]{SUV2019transmission}. Hence, we conclude that $R(\nabla_{tan}\phi)=0$

First note that the computation for $(a_R)_0$ in \eqref{zeroth_order_wave} is identical to the flat case and has no curvature terms.
For $(a_R)_{-1}$ we need to compute $P\phi_\bullet$ and $\p_\nu (a_T)_0$
which will have the curvature terms above. But to compute $(a_R)_{-2}$, we will need $\p_\nu (a_T)_{-1}$ which will involve $P\phi_T$ and $P (a_T)_0$. This will involve $\p_\nu^2 (a_T)_0$ which in turn involves $\p_\nu P\phi_T$ which will have second derivatives of the elastic parameters, first normal derivatives of curvature terms, and curvature terms. However, any term with curvature will be $R_1$ and known from the previous step.

Thus,
\begin{align*}
i c_S^2 \xi_{3,\bullet}\p_\nu(a_\bullet)_{-1}
&= \frac{1}{2\rho} (P\phi_\bullet)(a_\bullet)_{-1}
+ \tau \p_{t}(a_\bullet)_{-1} -c_S^2 \eta_{tan} \cdot \nabla_{tan}(a_\bullet)_{-1}
-P(a_\bullet)_0\\
&=-c_S^2 \left[(\p^2_{\nu}\log \sqrt \rho)
- (\p^2_{\nu}\log c_S)\left(1 - \frac{(\p_t \phi_\bullet)^2}{2c_S^2 \xi^2_{3,\bullet}}\right)+ \p_\nu H(x)/2 \right](a_\bullet)_0 \\
&\qquad \qquad + R_1
\end{align*}
where $R_1$ has at most one normal derivative of parameters and no normal derivatives of the curvature. The quantity $\p_\nu H$ is related to both the mean curvature and the Gauss curvature of $\Gamma$ \cite[Lemma 3.2]{DouganMean_Curv_derivative}. Again, any curvature term will be $R_1$ so that the main formulas remain the same.

 After iteration as in the previous section, we obtain
 \begin{multline}\label{e: (a)_J with higher curvature}
 \p_\nu(a_\bullet)_{J}
=\ \left(\frac{-i}{\xi_{3,\bullet}}\right)^{\abs{J}+1} \left[(\p^{\abs{J}+1}_{\nu}\log \sqrt \rho) \right.
\\\left.- (\p^{\abs{J}+1}_{\nu}\log c_S)\left(1 - \frac{(\p_t \phi_\bullet)^2}{2c_S^2 \xi^2_{3,\bullet}}\right)+ \p^{\abs{J}}_\nu H(x)/2 \right](a_\bullet)_0 + R_{\abs{J}},
\end{multline}
 where $R_{\abs{J}}$  includes up to $|J|-1$ normal derivatives of $H$. Using \eqref{e: transport equations} and \eqref{e: transmission conditions} with the same argument in the flat case, we arrive at the equation for $(a_R)_J$ in the statement of the Proposition. Lemma \ref{l: higher der of mean curvature in terms of S_Gamma} applied to $\p^{\abs{J}}_\nu H$ shows that all curvature terms can be determined from $S_\Gamma$.
 This implies that one only needs the shape operator (and not its derivatives!) to compute the full reflection operator.

In the elastic case as well, we may use boundary normal coordinates and this creates interface curvature terms as in \cite{RachBoundary}. However, these terms will contain one normal derivative less than the highest order normal derivatives of the material parameters, and would still be included in the $R_{\abs{J}}$ remainder terms. Hence, as in the above calculation for the acoustic case and in \cite{RachBoundary}, the same formulas hold in Proposition \ref{prop 5}, Proposition \ref{Prop_3.7}, Lemma \ref{Lem_3.12}, and Lemma \ref{l: higher order p speed from reflection} where $\p_{x_3}$ becomes $\p_\nu$. The remaining argument to prove Theorem \ref{t: Elastic case} proceeds as in the flat case in the previous section.
\end{proof}

  The higher order normal derivatives $\p_\nu^k H(x)$ can be related to the principal curvatures of the interface , using the methods of \cite{DouganMean_Curv_derivative}. Note this is irrelevant for Theorem \ref{Th_main_acoustic} since we just showed \eqref{e: p_x_3 (a)_J higher order} continues to hold even in the general case since $\p^{\abs{J}}_\nu H(x)$ are indeed $R_{\abs{J}}$ terms. Next, we show that even these higher order normal derivatives only depend on the curvature (shape operator) of the interface and not the higher order derivatives.

We follow \cite{DouganMean_Curv_derivative} to introduce a natural defining function for $\Gamma$ for the interface normal coordinates that we use to compute.
The signed distance function $b(x)$ to the surface $\Gamma$ is defined as
\[
b(x,\Gamma) = \begin{cases}
\text{dist}(x,\Gamma) & \mbox{for } x \in \Omega_- \\
 0 & \mbox{for } x \in \Gamma \\
-\text{dist}(x,\Gamma) & \mbox{for } x \in \Omega_+ \\
\end{cases}
\]
where
\[
\text{dist}(x,\Gamma) = \text{inf}_{y \in \Gamma}|y-x|.
\]
Then $\tilde x_3 = b$ is the defining function of $\Gamma$ and
\[
\nu = \nabla b(x)\restriction_\Gamma
\]
and we sometimes denote $\nu$ for the vector field $\nabla b = \nabla_x \tilde x_3$ where convenient. Since $b$ is a distance function, $\abs{\nabla b} = 1$ \cite{DouganMean_Curv_derivative}. Denote by $\kappa(x)$ at $\Gamma$ the mean curvature of $\Gamma$ at $x$ and $\kappa_i$ are the principal curvatures of the surface, which are the eigenvalues of $S_\Gamma$. As mentioned, $H(x)$ is proportional to $\kappa$ by a constant so that all our results for $\kappa$ extend naturally to $H$. We first mention the following important lemma
\begin{lemma}(\cite[Lemma 3.2]{DouganMean_Curv_derivative})\label{l: first der of mean curvature}
The normal derivative of the mean curvature of a surface $\Gamma$ of class $C^3$ only depends on the shape operator of $\Gamma$. More precisely
\[
\p_\nu \kappa = - \sum_i \kappa_i^2.
\]
For a two-dimensional surface in $3d$, this is equal to
\[
\p_\nu \kappa = - (\kappa_1^2 + \kappa_2^2) =
- (\kappa^2 - 2\kappa_G)
\]
where $\kappa_G = \kappa_1 \kappa_2$ denotes the Gauss curvature.
\end{lemma}

We shall extend this type of result to higher order derivatives as well.
\begin{lemma}
\label{l: higher der of mean curvature in terms of S_Gamma}
 All higher order normal derivative of the mean curvature of a surface $\Gamma$ of class $C^\infty$ only depends on the shape operator of $\Gamma$. More precisely
\[
\p^J_\nu \kappa = (-1)^J J! \sum_i \kappa_i^{J+1},
\]
and $\p_\nu^J H$ differs from this by a constant depending only on the dimension.
\end{lemma}
\begin{proof}
Observe that
\[
\p_\nu \kappa = (\nabla \kappa) \cdot \nabla b \restriction_\Gamma
\]
It is shown in \cite[Lemma 3.2]{DouganMean_Curv_derivative} that $\p_{\tilde x_3} \kappa = -\abs{D^2b}^2$ where $\abs{ \cdot }$ denotes the Frobenius norm of a matrix.

Thus
\begin{equation*}
-\p^2_{\tilde x_3} \kappa = \nabla (\abs{D^2 b}^2) \cdot \nabla b
= \nabla (b^2_{x_i x_j}) \cdot \nabla b
= 2b_{x_i x_j}b_{x_i x_j x_k}b_{x_k}
\end{equation*}

Next, we can use $1 = \abs{\nabla b}^2 = \sum_k b^2_{x_k} = b_{x_k} b_{x_k}$
where we understand the last equality as a sum over $k$, to obtain after a brief calculation
\begin{align}\label{e: partial^2 kappa with b_ij only}
\p_{\tilde x_3}^2 \kappa
&= 2 \text{tr}( (D^2 b)^3)
\end{align}
Next, $D^2 b\restriction_\Gamma = \nabla_\Gamma \nu$ whose eigenvalues are precisely the principal curvatures $\kappa_1, \dots, \kappa_{n-1}$ so that the eigenvalues of $D^2 b$ are $\kappa_i^3$ for $i = 1, \dots, n-1$.

Thus, for a constant $c_n$, we conclude
\[
\p_\nu^2 H = 2c_n \sum_i \kappa_i^3.
\]
We can obtain the higher order derivatives $\p^J_\nu \kappa$ analogously by using $1 = \abs{\nabla b}$ together with \eqref{e: partial^2 kappa with b_ij only} so that only terms in $D^2 b$ appear. In fact, we can show inductively that
\[
\p_{\tilde x_3}^J \kappa = (-1)^J J! \text{trace}((D^2 b)^{J+1}).
\]
Denote $b_{pq} = b_{qp} = b_{x_p x_q}$. Then by the inductive step
\[
\frac{1}{(-1)^{J-1} (J-1)!}\p_{\tilde x_3}^{J-1} \kappa =  \text{trace}((D^2 b)^{J}) = \sum_{i_1,\dots,i_{J}} b_{i_1 i_2} b_{i_2 i_3} \dots  b_{i_{J-1} i_{J}} b_{i_{J} i_1}
\]
 After a brief computation we obtain
\begin{align*}
\frac{1}{(-1)^{J-1} (J-1)!}\p_{\tilde x_3}^{J} \kappa  &= \nabla \left(\sum_{i_1,\dots,i_{J}} b_{i_1 i_2} b_{i_2 i_3} \dots  b_{i_{J-1} i_{J}} b_{i_{J} i_1} \right) \cdot \nabla b \\
&= -\sum_{p=1}^J \text{tr}(D^2b)^{J+1} = - J \text{tr}(D^2b)^{J+1}.
\end{align*}

Hence, using induction and taking the trace in the above formula allows us to conclude
\begin{equation}\label{e: partial^J H in terms of shape operator}
\p^J_\nu H = c_n(-1)^J J! \sum_i \kappa_i^{J+1}.
\end{equation}
\end{proof}

\section{An alternate viewpoint that relates to boundary determination}

As mentioned in the introduction, a relevant forward problem to our inverse problem is Michael Taylor's work in \cite{Taylor75}. Following \cite{yamamoto1989}, the elastic transmission problem may locally be cast as a first order boundary value problem near the interface, with $\Gamma$ acting as a boundary.
 Denote $\Omega_1 = \Omega_+$ and $\Omega_2 = \Omega_-$ from the introduction. We assume boundary normal coordinates are chosen so that locally, $\Gamma$ is given by $\{ x_3 = 0\}$. Again, $u$ denotes the solution to the elastic transmission problem on $\Omega \times \RR$ and we denote $u_i$ as $u$ restricted to $\Omega_i \times \RR$.

 Then, we denote $U_i = .^t(\Lambda(D_{x'},D_t) u_i, D_{x_3}u_i)$ where $\Lambda$ is a pseudo-differential operator with the symbol $\Lambda_1(\xi',\tau) = (\abs{\xi'}^2+\tau^2+1)^{1/2}.$
The transmission problem becomes the following boundary value problem
with the form (taken from \cite{yamamoto1989})
\[\begin{cases}
D_{x_3}U_i = M_i(x',D_{x'}, D_t) U_i \qquad & \text{ in }(-1)^{i+1}x_3 > 0\\
(I_3, 0 ) U_1 = (I_3, 0)U_2 \qquad & \text{ on } x_3 = 0, \\
B_1(x',D_{x'},D_t)U_1 = B_2(x',D_{x'}, D_t) U_2 \qquad & \text{ on } x_3 = 0
\end{cases}
\]
where $I_3$ is the $3 \times 3$ identity matrix, $M_i$ is a $3 \times 6$ matrix pseudo-differential operator of order one depending on the parameters in $\Omega_i$, and the $6 \times 3$ matrix principal symbol $(B_{i1}, B_{i2})(x',\xi',\tau)$ of $B_i = (B_{i1},B_{i2})(x',D_{x'}, D_t)$ is determined by the Neumann operator \eqref{elastic_Neumann} and depend on the parameters in region $\Omega_i$ (see \cite[Equation (2.2)]{yamamoto1989} for the exact definitions).

One may then construct the boundary operator $\gamma$ appearing in \cite{Taylor75} that determines a pseudodifferential equation between the ``incoming'' and ``outgoing'' elastic waves at the interface \cite{yamamoto1989}. The principal amplitudes of the outgoing waves at the interface are determined by $\gamma$, which are used to form the parametrix for the elastic wave equation away from glancing rays. Our inverse problem is to use these scattered amplitudes at the interface to determine the jet of the material parameters at a certain side of an interface.

\section{Declarations}
\subsection*{Funding} M.V.d.H. gratefully acknowledges support from the Simons Foundation under the MATH + X program, the National Science Foundation under grant
DMS-1815143, and the corporate members of the Geo-Mathematical Imaging Group at Rice University. G.U. was partly supported by NSF, a Walker Family Endowed Professorship at UW and a Si-Yuan Professorship at IAS, HKUST. S.B. was partly supported by Project no.: 16305018 of the Hong Kong Research Grant Council.

\subsection*{Conflict of interest/Competing interests}
{\bf Financial interests:} The authors declare they have no financial interests.
\\

\noindent {\bf Non-financial interests:} The authors declare they have no non-financial interests.

\subsection*{Availability of data and material} Not applicable

\subsection*{ Code availability} Not applicable

\appendix 

\section{Proofs of lemmas and propositions from Section \ref{Sec_Acoustic}}\label{s: proofs of acoustic lemmas}

These are proofs of the main statements in the acoustic case. Since they are a simpler, yet more lucid version of the elastic case, we relegate them to this appendix.

\subsection{Zeroth order recovery of the parameters at the interface}

\begin{proof}[Proof of Lemma \ref{l: recover 2 params from 0'th order reflect}]
By solving (\ref{zeroth_order_wave}), we get
\begin{equation*}
(a_R)_0(x',\tau,\xi') = \frac{\mu^{(-)}\xi_I - \mu^{(+)}\xi_T}{\mu^{(-)}\xi_I + \mu^{(+)}\xi_T}
= \frac{\mu^{(-)}\xi_I/\xi_T - \mu^{(+)}}{\mu^{(-)}\xi_I/\xi_T + \mu^{(+)}}
= \frac{af-c}{af+c},
\end{equation*}
where we denote $f = \xi_I/\xi_T, a = \mu^{(-)}, c = \mu^{(+)}.$ Note that $f = f(\abs{\xi'}/\tau)$ i.e. it is a function of the parameter $\abs{\xi'}/\tau$ while $a,c$ only depend on $x$.
Now, since $(a_R)_0 = (\tilde a_R)_0$ and assuming $\mu^{(-)} = \tilde{\mu}^{(-)}$ (i.e. $a=\tilde{a}$) on $\Gamma$ we obtain
\begin{align}\label{key_3}
\frac{af-c}{af+c} = \frac{a\tilde f-\tilde c}{a\tilde f+\tilde c},\quad
\text{if and only if} \quad a\tilde c f = ac \tilde f, \quad
\text{if and only if} \quad \frac{c}{\tilde c} = \frac{\tilde f}{f}.
\end{align}
Varying $\abs{\xi'}/\tau$ keeping everything else constant, we get
\begin{equation*}
\frac{c}{\tilde c} = \frac{\tilde f_1}{f_1},
\end{equation*}
where $f_1$ is $f$ evaluated at different value of $\abs{\xi'}/\tau$. Thus,
\begin{align*}
\frac{\tilde f}{f}=\frac{\tilde f_1}{f_1}\quad
\text{if and only if} \quad
\frac{(\tilde c^{(+)}_S)^{-2} - b^2}{(c^{(+)}_S)^{-2} - b^2}
= \frac{(\tilde c^{(+)}_S)^{-2} - b_1^2}{(c^{(+)}_S)^{-2} - b_1^2},
\end{align*}
where we used $c^{(-)}_S = \tilde c^{(-)}_S$ and labelled $b = \abs{\xi'}/\tau$, $b_1 = \abs{\xi'_1}/\tau_1$.
Cross multiplying we get the algebraic equation
\begin{equation*}
((\tilde c^{(+)}_S)^{-2}- ( c^{(+)}_S)^{-2})(b^2 - b_1^2) = 0
\end{equation*}
Note that, as long as we pick $b_1 \neq \pm b$, we recover $c_S^{(+)} = \tilde{c}_S^{(+)}$.
Then going back to (\ref{key_3}) one gets $c=\tilde{c}$, that is $\mu^{(+)} = \tilde \mu^{(+)}$ on $\Gamma$.
\end{proof}

\begin{rem} \label{rem: relative amplitudes computation}
In geophysical experiments, one often only has access to \emph{relative amplitudes} where the amplitude is normalized to be $1$ for an incident wave hitting the interface at a fixed particular angle. Concretely, for a fixed, $(x',\tau_1, \xi'_1)$ in the hyperbolic set, suppose one instead measures $R:= (a_R)_0(x',\tau,\xi')/ (a_R)_0(x',\tau_1,\xi'_1)$. The questions is can one recover $\mu^{(+)}$ and $c_S^{(+)}$ at $x'$ from $R$ at various incident angles (i.e. varying $\tau, \xi'$ in the hyperbolic set)?

In order to recover a single unknown parameter such as $\mu^{(+)}$, then this can be done with elementary means, but disentangling two material parameters is less clear.
For the uniqueness question, ignoring spacial variables, assume
\[\frac{(a_R)_0(\tau, \xi')}{(a_R)_0(\tau_1, \xi_1)}
=  \frac{(\widetilde{a_R})_0(\tau, \xi')}{(\widetilde{a_R})_0(\tau_1, \xi_1)}.\] Then one can show $\mu^{(+)} = \widetilde{\mu^{(+)}}$ and $c_S^{(+)} = \widetilde{c_S^{(+)}}$ with a similar argument as above.

One can even obtain a partial reconstruction algorithm directly from $R$. Via computation, one can show that \[L:= (R-1)/(R+1) = \frac{2ac(f-f_1)}{a^2ff_1 - c^2},\] where $a,c,f,f_1$ are as in the lemma. By solving a quadratic equation, we compute \[2c = \alpha/L + \sqrt{\alpha^2/L^2 + 4\beta}\] where $\alpha = 2a(f-f_1), \beta = a^2ff_1$ are independent of $c$. Since $c$ is independent of those variable, one can vary $\tau, \xi'$ within the hyperbolic set to obtain a nonlinear equation that needs to be solved for $c_S$ only (without any terms involving $c= \mu^{(+)}$), but it is unclear whether this can be done by elementary means. If it can, then our approach shows how one can do the recovery even with reflected amplitudes, and there is a reconstruction formula.
\end{rem}

\subsubsection{Recovery of the derivatives of the parameters at the interface}

\begin{proof}[Proof of Lemma \ref{l: acoustic partial_x_3 a_0 formula}]
Observe that for $J=0$, from (\ref{e: transport equations}) and (\ref{e: Hamilton derivative to normal deriv}) one obtains
\begin{equation}\label{key_6}
\p_{x_3}(a_\bullet)_0
= \frac{1}{2\rho c_S^2 \xi_{3, \bullet}}\left(P\phi_{\bullet}\right)(a_{\bullet})_0 + R_0.
\end{equation}
In order to calculate the term $P\phi_{\bullet} = (\rho\p_t^2 - \nabla_x\cdot\mu\nabla_x)\phi_{\bullet}$, we start with
\begin{equation}\label{key_4}
\p_{x_3} e^{i\phi_\bullet} = i\p_{x_3}\phi_\bullet e^{i\phi_\bullet}
= \pm \left(\sqrt{\abs{\p_{x'}\phi_{\bullet}}^2 -c^{-2}_{S}\abs{\p_{t}\phi_{\bullet}}^2}\right)e^{i\phi_\bullet},
\qquad \mbox{on }\Gamma.
\end{equation}
Therefore, $\p_{x_3} \phi_\bullet = \pm \sqrt{c^{-2}_{S}\abs{\p_{t}\phi_{\bullet}}^2 - \abs{\p_{x'}\phi_{\bullet}}^2}$ can be recovered from $c_S$ and the tangential derivatives of $\phi_{\bullet}$ on $\Gamma$.
In other words, $\p_{x_3} \phi_\bullet$ at $\Gamma$ is a $R_0$ term.
Taking one more derivative of (\ref{key_4}) in the normal direction we get
\begin{equation*}
\p^2_{x_3}\phi_\bullet = \frac{1}{\xi_{3, \bullet}}\left[-c_S^{-2}(\p_{x_3}\log c_S)\abs{\p_t \phi_\bullet}^2 + c_S^{-2}(\p_t \phi_\bullet)\p_{x_3}\p_t \phi_\bullet
- \langle \p_{x_3}\nabla_{x'}\phi_\bullet, \nabla_{x'}\phi_\bullet \rangle\right].
\end{equation*}
Here the last two terms above are determined by $c_S$, $\phi_{\bullet}$ and their tangential derivatives on $\Gamma$. Hence, we have
\begin{equation*}
\p^2_{x_3}\phi_\bullet =-(\p_{x_3}\log c_S)c_S^{-2}\xi_{3, \bullet}^{-1}\abs{\p_t \phi_\bullet}^2 + E_0, \quad \mbox{where }E_0\mbox{ is a }R_0\mbox{ term}.
\end{equation*}
If we take one more normal derivative of $\phi_{\bullet}$, then $\p_{x_3}E_0$ can have at most one derivative of $c_S$ as well as the term $c_S^{-2}\xi_{3, \bullet}^{-1}\abs{\p_t \phi_\bullet}^2$.
Thus, we obtain
\begin{equation*}
\p_{x_3}^3 \phi_\bullet = -(\p^2_{x_3}\log c_S)c_S^{-2}\xi_{3, \bullet}^{-1}\abs{\p_t \phi_\bullet}^2 + R_1.
\end{equation*}
In general, one obtains
\begin{equation}\label{normal_phi_k_wave}
\p_{x_3}^k \phi_\bullet = -(\p^{k-1}_{x_3}\log c_S)c_S^{-2}\xi_{3, \bullet}^{-1}\abs{\p_t \phi_\bullet}^2 + R_{k-2}.
\end{equation}
%
Now we calculate
\begin{align*}
\nabla_x \cdot \mu \nabla_x \phi_\bullet &= (\p_{x_3}\mu) \p_{x_3} \phi_\bullet + \mu \p^2_{x_3}\phi_\bullet + R_0\\
&= \p_{x_3}(\rho c_S^2)\xi_{3, \bullet}- \rho c_S^2(\p_{x_3}\log c_S)c_S^{-2}\xi_{3, \bullet}^{-1}(\p_t \phi_\bullet)^2 +R_0
\\
&= \rho( (\p_{x_3}\log \rho)c_S^2\xi_{3,\bullet}
+ (\p_{x_3}\log c_S)(2c_S^2\xi_{3,\bullet} - \xi_{3, \bullet}^{-1}(\p_t \phi_\bullet)^2)) + R_0.
\end{align*}
Also note that $\p^2_t \phi_{\bullet} = \p^2_t\left(-\tau t + x'\cdot\xi'\right) = 0$ on $\Gamma$.
Thus, from a direct calculation, we obtain
\begin{align*}
(1/2\rho)P\phi_{\bullet}
=& \frac{1}{2}\p_t^2 \phi_{\bullet} - \nabla_x\cdot\mu\nabla_x \phi_{\bullet}\\
=& -(1/2)( (\p_{x_3}\log \rho)c_S^2\xi_{3,\bullet}
+ (\p_{x_3}\log c_S)(2c_S^2\xi_{3,\bullet} - \xi_{3, \bullet}^{-1}(\p_t \phi_\bullet)^2)) + R_0,\\
&\quad \text{on}\quad \Gamma.
\end{align*}
Therefore, going back to \eqref{key_6} we get
\begin{align*}
c_S^2&\xi_{3,\bullet}\p_{x_3}(a_\bullet)_0\\
=& -(1/2)( (\p_{x_3}\log \rho)c_S^2\xi_{3,\bullet}
+ (\p_{x_3}\log c_S)(2c_S^2\xi_{3,\bullet} - \xi_{3, \bullet}^{-1}(\p_t \phi_\bullet)^2))(a_\bullet)_0 + R_0,\\
\end{align*}
so that
\begin{equation*}
\p_{x_3}(a_\bullet)_0
= -\left[(\p_{x_3}\log \sqrt \rho)
- \p_{x_3}\log c_S\left(1 - \frac{(\p_t \phi_\bullet)^2}{2c_S^2 \xi^2_{3,\bullet}}\right)\right](a_\bullet)_0 + R_0.
\end{equation*}
\end{proof}

\begin{rem}\label{rem:1}
	Observe that $\p_{x_3}(a_{I})_0$ and $\p_{x_3}(a_{R})_0$ are indeed $R_0$ terms, because $\rho^{(-)}$ and $c_{S}^{(-)}$ are known on the $\Omega_{-}$ region and so are $\p_{x_3}\rho^{(-)}$, $\p_{x_3}c_{S}^{(-)}$ on $\Gamma$. The rest of the terms in the expression of $\p_{x_3}(a_{I})_0$ and $\p_{x_3}(a_{R})_0$ can be determined from the $0$-th order transmission condition (\ref{zeroth_order_wave}). On the other hand $\p_{x_3}(a_{T})_0$ is not $R_0$ but $R_1$ since it involves $\rho^{(+)}$ and $c_{S}^{(+)}$, which cannot be determined from (\ref{zeroth_order_wave}).
\end{rem}

\begin{proof}[Proof of Lemma \ref{l: first derivative}]
We start with the transmission conditions for $(a_R)_{-1}$.
From (\ref{e: transmission conditions}) for $J=-1$, we get

\begin{align}\label{key}
(a_R)_{-1} =& \frac{1}{\mu^{(-)}\xi_{3, R} + \mu^{(+)}\xi_{3, T}}\left[\mu^{(-)}\p_{x_3}(a_I)_0 + \mu^{(-)}\p_{x_3}(a_R)_0 - \mu^{(+)}\p_{x_3}(a_T)_0 \right].
\end{align}

Note that $\mu^{(+)}$ can be determined by the $0$-th order transmission condition (see Lemma \ref{l: recover 2 params from 0'th order reflect}), therefore, $\left(\mu^{(-)}\xi_{3, R} + \mu^{(+)}\xi_{3, T}\right)^{-1}$ is a $R_0$ quantity.
Furthermore, thanks to Lemma \ref{l: acoustic partial_x_3 a_0 formula}, $\mu^{(-)}\p_{x_3}(a_I)_0$ and $\mu^{(-)}\p_{x_3}(a_R)_0$ are $R_0$, see Remark \ref{rem:1}.
From \eqref{key} and \eqref{e: acoustic partial_{x_3}a_0 formula} one obtains
\begin{equation*} (a_R)_{-1}
=\ R_0 + \left[\left(\p_{x_3}\log \sqrt {\rho^{(+)}}\right)
- \p_{x_3}\log c_S^{(+)}\left(1 - \frac{(\p_t \phi_T)^2}{2(c_S^{(+)})^2 \xi^2_{3,T}}\right)\right]\frac{(a_T)_0}{R_0}.
\end{equation*}

 We denote $f=f(\abs{\xi'}/\tau) = \left(1 - \frac{(\p_t \phi_T)^2}{2(c_S^{(+)})^2 \xi^2_{3,T}}\right)$. If we have $\tilde{\rho}^{(-)} = \rho^{(-)}$, $\tilde{\mu}^{(-)} = \mu^{(-)}$ on $\Omega_{-}$ and $\tilde{R} = R$ on $\Gamma$, then one gets $\tilde{R}_0 = R_0$ and $(\tilde{a}_R)_{-1} = (a_R)_{-1}$ on $\Gamma$.
Therefore, we obtain

\begin{align*}
R_0 + &\left[\left(\p_{x_3}\log \sqrt {\rho^{(+)}}\right)
- \p_{x_3}\log c_S^{(+)}\left(1 - \frac{(\p_t \phi_T)^2}{2(c_S^{(+)})^2 \xi^2_{3,T}}\right)\right]\frac{(a_T)_0}{R_0}\\
=& R_0 + \left[\left(\p_{x_3}\log \sqrt {\tilde{\rho}^{(+)}}\right)
- \p_{x_3}\log \tilde{c}_S^{(+)}\left(1 - \frac{(\p_t \tilde{\phi}_T)^2}{2(\tilde{c}_S^{(+)})^2 \tilde{\xi}^2_{3,T}}\right)\right]\frac{(a_T)_0}{R_0},\qquad \mbox{on}\quad \Gamma
\end{align*}

This implies
\begin{equation}\label{key_8}
\qquad
\left(\p_{x_3}\log \sqrt {\frac{\rho^{(+)}}{\tilde{\rho}^{(+)}}}\right)
=\left(\p_{x_3}\log c_S^{(+)}\right)f - \left(\p_{x_3}\log \tilde{c}_S^{(+)}\right)\tilde{f},\qquad \mbox{on}\quad \Gamma.
\end{equation}
Observe that $f$ is a $R_0$ quantity so that $f=\tilde{f}$.
Furthermore, $\rho^{(+)}$, $\tilde{\rho}^{(+)}$ depends only on $x$, hence by varying $(\abs{\xi'}/\tau)$ we obtain
\begin{equation}\label{key_7}
\left(\p_{x_3}\log \frac{c_S^{(+)}}{\tilde{c}_S^{(+)}}\right)f
= \left(\p_{x_3}\log \frac{c_S^{(+)}}{\tilde{c}_S^{(+)}}\right)f_1,
\qquad \mbox{on}\quad \Gamma,
\end{equation}
where $f$ and $f_1$ are evaluated in different values of $\abs{\xi'}/\tau$.
Note that $c_S^{(+)} = \tilde{c}_S^{(+)}$ on $\Gamma$ (see Lemma \ref{l: recover 2 params from 0'th order reflect}).
If we take two values of $\abs{\xi'}/\tau$ such a way that $f \neq f_1$ on $\Gamma$, then \eqref{key_7} implies
\begin{equation*}
\left(\p_{x_3}\log \frac{c_S^{(+)}}{\tilde{c}_S^{(+)}}\right)
= 0,
 \ \ \mbox{    so that    } \ \
\p_{x_3}c_S^{(+)} = \p_{x_3}\tilde{c}_S^{(+)} \quad \mbox{on }\Gamma.
\end{equation*}
Going back to \eqref{key_8} we obtain $\p_{x_3}\rho^{(+)} = \p_{x_3}\tilde{\rho}^{(+)}$ and thus $\p_{x_3}\mu^{(+)} = \p_{x_3}\tilde{\mu}^{(+)}$ on $\Gamma$.
%
\end{proof}

\begin{proof}[Proof of Lemma \ref{l: acoustic higher derivatives derivation}]
We prove this lemma via an iterative argument.
First we note that for $J=0,-1$ we already have Lemma \ref{l: first derivative} and Lemma \ref{l: first derivative}.

In order to prove the lemma for $J<-1$ we study the transport equation \eqref{e: transport equations}.
For $J<0$, in the transport equations \eqref{e: transport equations}, we encounter the term $P(t,x,D_{t,x})(a_\bullet)_J = \rho\p_t^2(a_\bullet)_J - \nabla \cdot \mu \nabla (a_\bullet)_J$.
We calculate
\begin{align*}
\nabla \cdot \mu \nabla (a_\bullet)_0
=& (\p_{x_3}\mu) \p_{x_3}(a_\bullet)_0  + \mu \p^2_{x_3}(a_\bullet)_0 + R_0\\
=& \mu \p^2_{x_3}(a_\bullet)_0 + R_1 \\
=&-\mu \left[(\p^2_{x_3}\log \sqrt \rho)
 +\p^2_{x_3}\log c_S\left(1 - \frac{(\p_t \phi_\bullet)^2}{2c_S^2 \xi^2_{3,\bullet}}\right)\right](a_\bullet)_0\\
&-\mu\left[(\p_{x_3}\log \sqrt \rho)
 +\p_{x_3}\log c_S\left(1 - \frac{(\p_t \phi_\bullet)^2}{2c_S^2 \xi^2_{3,\bullet}}\right)\right]^2(a_\bullet)_0
+R_1
\end{align*}
Using the equation earlier for $\p_{x_3}(a_\bullet)_0$, we see the second term is in fact $R_1$. Thus, we obtain
\begin{equation}\label{key_9}
P(t,x,D_{t,x})(a_\bullet)_0
= \mu \left[(\p^2_{x_3}\log \sqrt \rho)
 +\p^2_{x_3}\log c_S\left(1 - \frac{(\p_t \phi_\bullet)^2}{2c_S^2 \xi^2_{3,\bullet}}\right)\right](a_\bullet)_0 + R_1
\end{equation}
Now, from the transport equation \eqref{e: transport equations} and the relation \eqref{key_5} ,\eqref{key_9} we get
\begin{equation}\label{key_10}
\p_{x_3}(a_\bullet)_{-1}
= -i/(2\xi_{3,\bullet}) \left[(\p^2_{x_3}\log \sqrt \rho)
 +\p^2_{x_3}\log c_S\left(1 - \frac{(\p_t \phi_\bullet)^2}{2c_S^2 \xi^2_{3,\bullet}}\right)\right](a_\bullet)_0 + R_1.
\end{equation}

Since the microlocal transmission conditions \eqref{e: transmission conditions} helps us to connect $(a_R)_{-2}$ to $\p_{x_3}(a_T)_{-1}$,
using the same argument as in Lemma \ref{l: first derivative} we see that $(a_R)_{-2}$ uniquely determines $\p^2_{x_3} \rho^{(+)}$ and $\p^2_{x_3} \mu^{(+)}$ at $\Gamma$.
Iterating the above argument gives us
\begin{equation}\label{e: p_x_3 (a)_J higher order}
\p_{x_3}(a_\bullet)_J
= (-i/(2\xi_{3,\bullet}))^J \left[(\p^{\abs{J}+1}_{x_3}\log \sqrt \rho)
 +\p^{\abs{J}+1}_{x_3}\log c_S\left(1 - \frac{(\p_t \phi_\bullet)^2}{2c_S^2 \xi^2_{3,\bullet}}\right)\right](a_\bullet)_0 + R_{\abs{J}}.
\end{equation}

Then we get from the $|J|$-th order transmission conditions
\[
\p_{x_3}(a_T)_{J+1} = R_{\abs{J+1}}(R_{\abs{J+1}} - (a_R)_{J}) \text{ on } \Gamma
\]
so that
\begin{align}
(a_R)_J &= -\frac{\p_{x_3}(a_T)_{J+1}}{R_{\abs{J+1}}} + R_{\abs{J+1}} \nonumber \\
&= -(-i/(2\xi_{3,T}))^J \left[(\p^{\abs{J}}_{x_3}\log \sqrt {\rho^{(+)}}) \right.\\ \nonumber
&\qquad \qquad \qquad
 \left.+\p^{\abs{J}}_{x_3}\log c^{(+)}_S\left(1 - \frac{(\p_t \phi_T)^2}{2c_S^2 \xi^2_{3,T}}\right)\right]\frac{(a_T)_{J+1}}{R_{\abs{J+1}}} + R_{\abs{J+1}}
\end{align}

Using the same argument as above, and noting that the transmission conditions already determine $(a_T)_{J+1}$ from knowledge of $(a_R)_{J+1}$, shows that $(a_R)_J$ determines $\p^{\abs{J}}_{x_3} \rho^{(+)}$ and $\p^{\abs{J}}_{x_3} \mu^{(+)}$ at $\Gamma$.
\end{proof}

This completes the proof of Theorem \ref{Th_main_acoustic}.
The essential piece to make this work is verifying that $(a_R)_J$ at $\Gamma$ only depends on at most $|J|$ normal derivatives of the material parameters using the transmission conditions to continue unique recovery inductively.
We finish this section by the following remark.
\begin{rem}
Note that the recovery of the parameters on the boundary is obtained directly from the principal symbol of the reflection operator $R$, whereas recovering the higher order derivatives one relies on the recursive equations obtained from the interface conditions on the asymptotes of the geometric optics solutions.
\end{rem}

 \bibliography{ScatteringControl}
\end{document}